\documentclass{amsart}
\usepackage[all]{xy}
\usepackage{verbatim}
\usepackage{color}
\usepackage{amsthm}
\usepackage{amssymb}
\usepackage[colorlinks=true]{hyperref}

\numberwithin{equation}{section}

\newtheorem{theorem}[equation]{Theorem}
\newtheorem*{theorem*}{Theorem} \newtheorem{lemma}[equation]{Lemma}

\newtheorem*{conjecture*}{Mamma Conjecture}
\newtheorem*{conjecture1*}{Mamma Conjecture (revisited)}
\newtheorem{proposition}[equation]{Proposition}
\newtheorem{corollary}[equation]{Corollary}
\newtheorem*{corollary*}{Corollary}

\theoremstyle{remark}
\newtheorem{definition}[equation]{Definition}

\newtheorem{example}[equation]{Example}

\newtheorem{notation}[equation]{Notation}

\theoremstyle{remark}
\newtheorem{remark}[equation]{Remark}

\newcommand{\cA}{{\mathcal A}}
\newcommand{\cB}{{\mathcal B}}
\newcommand{\cC}{{\mathcal C}}
\newcommand{\cD}{{\mathcal D}}

\newcommand{\cF}{{\mathcal F}}

\newcommand{\cL}{{\mathcal L}}

\newcommand{\cN}{{\mathcal N}}
\newcommand{\cO}{{\mathcal O}}
\newcommand{\cP}{{\mathcal P}}

\newcommand{\cW}{{\mathcal W}}
\newcommand{\cX}{{\mathcal X}}

\newcommand{\cZ}{{\mathcal Z}}

\newcommand{\RHom}{{\mathrm{RHom}}}

\newcommand{\bbF}{\mathbb{F}}
\newcommand{\bbG}{\mathbb{G}}

\newcommand{\bbN}{\mathbb{N}}

\newcommand{\bbQ}{\mathbb{Q}}
\newcommand{\bbZ}{\mathbb{Z}}

\DeclareMathOperator{\SmProj}{SmProj} 

\DeclareMathOperator{\Id}{Id}
\DeclareMathOperator{\id}{id}

\DeclareMathOperator{\NChow}{NChow} 
\DeclareMathOperator{\NNum}{NNum} 
\DeclareMathOperator{\CSep}{CSep}
\DeclareMathOperator{\Sep}{Sep}
\DeclareMathOperator{\CSA}{CSA} 
\DeclareMathOperator{\AM}{AM} 
\DeclareMathOperator{\Cov}{Cov} 
\DeclareMathOperator{\Heck}{Heck} 
\DeclareMathOperator{\Hecke}{Hecke} 
\DeclareMathOperator{\Perm}{Perm} 
\DeclareMathOperator{\Etale}{Etale} 

\DeclareMathOperator{\Ch}{Ch}
\DeclareMathOperator{\Fun}{Fun} 

\newcommand{\dgcat}{\mathrm{dgcat}} 

\newcommand{\perf}{\mathrm{perf}}

\newcommand{\Chow}{\mathrm{Chow}}

\newcommand{\dg}{\mathrm{dg}}

\newcommand{\Hom}{\mathrm{Hom}}
\newcommand{\End}{\mathrm{End}}

\newcommand{\rep}{\mathrm{rep}}

\newcommand{\Hmo}{\mathrm{Hmo}}
\newcommand{\op}{\mathrm{op}}

\newcommand{\Map}{\mathrm{Map}}

\newcommand{\too}{\longrightarrow}

\newcommand{\REnd}{\mathbf{R}\mathrm{End}}

\newcommand{\ie}{\textsl{i.e.}\ }
\newcommand{\eg}{\textsl{e.g.}}

\def\Perm{\operatorname{Perm}}
\def\Cov{\operatorname{Cov}}
\def\CSep{\operatorname{CSep}}
\def\Sep{\operatorname{Sep}}

\def\ind{\operatorname{ind}}

\def\Br{\operatorname{Br}}
\def\dBr{\operatorname{dBr}}
\def\r{\rightarrow}

\let\oldmarginpar\marginpar
\def\marginpar#1{\oldmarginpar{\tiny #1}}
\def\DPic{\operatorname{DPic}}
\def\Pic{\operatorname{Pic}}
\def\pr{\operatorname{pr}}

\def\cosk{\operatorname{cosk}}

\def\Spec{\operatorname{Spec}}

\begin{document}

\title[Noncommutative motives of separable algebras]{Noncommutative motives of separable algebras}
\author{Gon{\c c}alo~Tabuada and Michel Van den Bergh}

\address{Gon{\c c}alo Tabuada, Department of Mathematics, MIT, Cambridge, MA 02139, USA}
\email{tabuada@math.mit.edu}
\urladdr{http://math.mit.edu/~tabuada}
\thanks{G.~Tabuada was partially supported by a NSF CAREER Award.}

\address{Michel Van den Bergh, Departement WNI, Universiteit Hasselt, 3590 Diepenbeek, Belgium}
\email{michel.vandenbergh@uhasselt.be} 
\urladdr{http://hardy.uhasselt.be/personal/vdbergh/Members/~michelid.html}
\thanks{M.~Van den Bergh is a Director of Research at the FWO Flanders}

\date{\today}


\abstract{In this article we study in detail the category of
  noncommutative motives of separable algebras $\mathrm{Sep}(k)$ over
  a base field $k$. We start by constructing four different models of
  the full subcategory of commutative separable algebras
  $\mathrm{CSep}(k)$. Making use of these models, we then explain how the category $\mathrm{Sep}(k)$
  can be described as a ``fibered $\bbZ$-order'' over
  $\mathrm{CSep}(k)$. This viewpoint leads to several computations and
  structural properties of the category $\mathrm{Sep}(k)$. For
  example, we obtain a complete dictionary between directs sums of
  noncommutative motives of central simple algebras (=CSA) and
  sequences of elements in the Brauer group of $k$. As a first
  application, we establish two families of motivic relations between
  CSA which hold for every additive invariant (\eg\ algebraic
  $K$-theory, cyclic homology, and topological Hochschild
  homology). As a second application, we compute the additive
  invariants of twisted flag varieties using solely the Brauer classes
  of the corresponding CSA. Along the way, we categorify the cyclic
  sieving phenomenon and compute the (rational) noncommutative motives
  of purely inseparable field extensions and of dg Azumaya algebras.}
}

\maketitle
\vskip-\baselineskip
\vskip-\baselineskip



\section{Introduction}
\subsection*{Noncommutative motives}
A {\em dg category} $\cA$, over a base field $k$, is a category enriched over complexes of $k$-vector spaces; see \S\ref{sec:dg}. Every (dg) $k$-algebra $A$ naturally gives rise
to a dg category with a single object. Another source of examples is provided by schemes since the category of perfect complexes $\perf(X)$ of every quasi-compact quasi-separated $k$-scheme $X$ admits a canonical dg enhancement $\perf_\dg(X)$. In what follows, $\dgcat(k)$ denotes the category of dg~categories.

Invariants such as algebraic $K$-theory, cyclic homology, and topological Hochschild homology, extend naturally from $k$-algebras to dg categories. In order to study them simultaneously the notion of {\em additive invariant} was introduced in \cite{Additive} and the {\em universal additive invariant} $U:\dgcat(k) \to \Hmo_0(k)$ was constructed; consult \S\ref{sub:additive}-\ref{sub:universal} for details. Given any additive category $\mathrm{D}$, there is an induced equivalence
\begin{equation}\label{eq:categories}
U^\ast: \Fun(\Hmo_0(k),\mathrm{D}) \stackrel{\simeq}{\too} \Fun_{\mathrm{add}}(\dgcat(k),\mathrm{D})\,,
\end{equation}
where the left-hand-side denotes the category of additive functors and the right-hand-side the category of additive invariants. Because of this universal property, which is reminiscent from the yoga of motives, $\Hmo_0(k)$ is called the category of {\em noncommutative motives}. The tensor product of $k$-algebras extends also naturally to dg categories. It gives rise to a symmetric monoidal structure on $\dgcat(k)$ which descends to $\Hmo_0(k)$ making the functor $U$ symmetric monoidal.

Following Kontsevich \cite{IAS,Miami,finMot}, a dg category $\cA$ is called {\em smooth} if it is perfect as a bimodule over itself and {\em proper} if $\sum_i \mathrm{dim} \,H^i \cA(x,y)< \infty$ for every ordered pair of objects $(x,y)$. Examples include the finite dimensional $k$-algebras of finite global dimension (when $k$ is perfect) and the dg categories $\perf_\dg(X)$ associated to smooth projective $k$-schemes $X$. The category of {\em noncommutative Chow motives $\NChow(k)$} was introduced in \cite{CvsNC} as the idempotent completion of the full subcategory of $\Hmo_0(k)$ consisting of the smooth proper dg categories. By construction, $\NChow(k)$ is additive and rigid symmetric monoidal; consult the survey \cite[\S4]{Buenos}.
\subsection*{Motivating goal}
Given an additive rigid symmetric monoidal category $\cC$, its $\otimes$-ideal $\cN$ is defined by the following formula
$$ \cN(a,b):=\{f:a \to b\,|\, \forall \,\,g:b \to a \,\,\,\text{we}\,\,\text{have}\,\,\, \mathrm{tr}(g \circ f) =0\}\,,$$
where $\mathrm{tr}$ stands for the categorical trace. Motivated by Andr{\'e}-Kahn's description of the category of numerical motives (see \cite[Example~7.1.2]{AK}), the category of {\em noncommutative numerical motives $\NNum(k)$} was introduced in \cite{Semisimple} as the idempotent completion of the quotient category $\NChow(k)/\cN$. Our first result is the following:
\begin{theorem}\label{thm:motivation}
When $k$ is of characteristic zero, the Hom-groups of the additive category $\NNum(k)$ are finitely generated abelian groups.
\end{theorem}
Intuitively speaking, Theorem \ref{thm:motivation} shows that the category $\NNum(k)$ encode only a finite amount of data. This motivates the following ambitious goal:

\vspace{0.1cm}

{\it Goal: Construct a simple and explicit model of the category $\NNum(k)$.}

\vspace{0.1cm}

Since $\NNum(k)$ contains information about {\em all} smooth proper dg categories, the above goal seems completely out of reach at
the present time. In this article we give the first step towards its
solution by addressing the case of the full subcategory of separable
algebras. Already in this case some surprisingly interesting phenomena
occur! For example, the latter category is strongly related with the classical theory of $\bbZ$-orders; see Proposition \ref{prop:ring}. For a number of applications, please consult \S\ref{sec:applications}.  
\section{Statement of results}
Throughout the article, except in the above Theorem \ref{thm:motivation}, $k$ will be a field of arbitrary
characteristic. Unless stated differently, all tensor products will be
taken over $k$. Let $G:=\mathrm{Gal}(k_{\mathrm{sep}}/k)$ be the
absolute Galois group of $k$, equipped with its profinite
  topology. Given a central simple $k$-algebra $A$, we will write $\mathrm{deg}(A)$ for its degree, $\mathrm{ind}(A)$ for its index, $\mathrm{per}(A)$ for its period, and
$[A]$ for its class in the Brauer group $\mathrm{Br}(k)$ of $k$.
\subsection*{Commutative separable algebras}
Recall from \cite[\S III Prop.~4.1]{Knus} that a commutative $k$-algebra $A$ is separable if and only if it is isomorphic to a finite product of finite separable field extensions of $k$. Thanks to \cite[\S III Thm.~1.4(1) and Prop.~3.2]{Knus}, every (commutative) separable $k$-algebra is smooth and proper as a dg category. Let us then denote by $\CSep(k)$ the full subcategory of $\NChow(k)$ consisting of the objects $U(A)$ with $A$ a commutative separable $k$-algebra. As mentioned in \S\ref{sub:universal}, $U(\cA) \oplus U(\cB) \simeq U(\cA \times \cB)$ for all dg categories $\cA$ and $\cB$. This implies that the category $\CSep(k)$ is additive. Since (commutative) separable $k$-algebras are stable under tensor product (see \cite[\S III Prop.~1.7]{Knus}) and the universal additive invariant $U$ is symmetric monoidal, $\CSep(k)$ is moreover symmetric monoidal. The following result is a special case of Proposition \ref{prop:properties2} stated below.
\begin{proposition}\label{prop:properties} 
The quotient functor $\CSep(k) \to\NNum(k)$ is fully-faithful.
\end{proposition}
Intuitively speaking, Proposition~\ref{prop:properties} shows that $\CSep(k)$ is insensitive to the numerical equivalence relation. Consider now the following four categories:
\begin{itemize}
\item[(i)] Recall from \cite[\S4]{Andre} that the category of Chow motives $\Chow(k)$ comes equipped with a symmetric monoidal functor $M: \SmProj(k)^\op \to \Chow(k)$ defined on smooth projective $k$-schemes. Let us denote by $\Etale(k)$ the full subcategory of $\Chow(k)$ consisting of the objects $M(X)$ with $X$ a finite {\'e}tale $k$-scheme. Note that $\Etale(k)$ is an additive symmetric monoidal category.
\item[(ii)] Given a finite $G$-set $S$, let $\Map^G(S,\bbZ)$ be the set of $G$-invariant functions $\alpha: S \to \bbZ$. The {\em convolution category $\Cov(G)$} is defined as follows: 
the objects are the finite $G$-sets $S$; the morphisms $\Hom_{\Cov(G)}(S_1,S_2)$ are the $G$-invariant functions $\Map^G(S_1 \times S_2, \bbZ)$; the composition law is the convolution product
$$
\Map^G(S_1\times S_2, \bbZ) \times \Map^G(S_2\times S_3, \bbZ)  \too \Map^G(S_1 \times S_3, \bbZ)
$$
where 
$$
(\alpha, \beta) \mapsto (\alpha \ast \beta) (s_1, s_3):= \sum_{s_2 \in S_2}\alpha(s_1,s_2) \cdot \beta(s_2, s_3)\,;
$$
the identities are the $G$-invariant functions $S \times S \to \bbZ$ which are equal to $1$ on the diagonal and $0$ elsewhere. The disjoint union and the cartesian product of finite $G$-sets makes $\Cov(G)$ into an additive symmetric monoidal category.
\item[(iii)] Given closed subgroups $H,K \subseteq G$ of finite index, let $H \backslash G/K$ be the set of $(H,K)$ double cosets in $G$ and $\Map(H\backslash G/K, \bbZ)$ the set of functions $\alpha: H\backslash  G/K \to \bbZ$. The category $\Heck(G)$ is defined as follows: the objects are the closed subgroups $H \subseteq G$ of finite index; the morphisms $\Hom_{\Heck(G)}(H,K)$ are the functions $\Map(H\backslash  G/K, \bbZ)$; the composition law is the convolution~product
$$ \Map(H\backslash G/ K, \bbZ) \times \Map(K\backslash G/ L, \bbZ) \too \Map(H\backslash G/ L, \bbZ) \,,$$
where 
$$ (\alpha, \beta) \mapsto (\alpha \star\beta)(\overline{g}) := \sum_{\overline{h} \in G/K}\alpha(h^{-1}) \cdot \beta(h g)\,;$$
the identities are the characteristic functions $\delta_{H1H}: H\backslash G/H \to \bbZ$. Note that $\End_{\Heck(G)}(H,H)$ identifies with the classical Hecke algebra of the pair $(G,H)$. In the particular case where $H$ is a normal subgroup of $G$, the latter Hecke algebra reduces to the group ring $\bbZ[G/H]$. The {\em Hecke category $\Hecke(G)$} is defined as the closure of $\Heck(G)$ under (formal) finite direct sums. By construction, $\Hecke(G)$ is an additive category.
\item[(iv)] A right $\bbZ[G]$-module $N$ is called a {\em permutation $G$-module} if it admits a finite $\bbZ$-basis which is invariant under the $G$-action. Equivalently, $N$ is of the form $\bbZ[S]$ for some finite $G$-set $S$.  Let us denote by $\Perm(G)$ the full subcategory of those right $\bbZ[G]$-modules which are permutation $G$-modules. Note that $\Perm(G)$ is an additive category.
\end{itemize}
By combining Merkurjev-Panin's work \cite[Prop.~1.7]{MP} with \cite[Thm.~6.10]{twisted}, we obtain the following additive equivalence of categories
\begin{eqnarray}\label{eq:equivalence-MP}
\CSep(k) \stackrel{\simeq}{\too} \Perm(G) && U(A) \mapsto K_0(A \otimes k_{\mathrm{sep}})\,.
\end{eqnarray}
Other models of the category $\CSep(k)$ are provided by the following result:
\begin{proposition}\label{thm:main1}
The following additive categories are equivalent:
\begin{eqnarray}\label{eq:5-categories}
\CSep(k) \quad \Etale(k) \quad \Cov(G) \quad \Hecke(G) \quad \Perm(G)\,.
\end{eqnarray}
Their idempotent completion is the category of (integral) Artin motives $\AM(k)$.
\end{proposition}
By combining Proposition \ref{thm:main1} with some examples arising from integral representation theory, we obtain the following (surprising) remarks:
\begin{remark} 
\label{rem:swan}
Let $H:=\langle a, b, c\,|\, a^8=b^2=c^2=abc\rangle$ be the generalized quaternion group of order $32$. R. Swan constructed in \cite{Swan}  a non-free projective left ideal $I \subset \bbZ[H]$ such that $I \oplus \bbZ[H] \simeq \bbZ[H] \oplus \bbZ[H]$ as $\bbZ[H]$-modules. We claim that $I$ is {\em not} a permutation $H$-module. Assume that $I \simeq \bbZ[S]$ for some finite $H$-set $S$. Since $H$ is a $2$-group, $(\bbZ/2\bbZ)[H]$ is a local ring. This implies that $I$ is indecomposable and consequently that $S$ has a single $H$-orbit. Making use of the equality $\mathrm{rank}_\bbZ(I)=\mathrm{rank}_\bbZ(\bbZ[H])$, we hence conclude that $S\simeq H$. This contradicts the fact that $I$ is non-free. As an application, we obtain the following~results:
\begin{itemize}
\item[(i)] The above categories \eqref{eq:5-categories} are {\em not} idempotent complete! Choose a Galois field extension $l/k$ inside $k_{\mathrm{sep}}$ with Galois group $H$. Via the induced group homomorphism $G \twoheadrightarrow H$, $\mathrm{Perm}(H)$ identifies with a full subcategory of $\Perm(G)$. Consequently, the idempotent in $\End_{\Perm(G)}(\bbZ[H]\oplus \bbZ[H])$ corresponding to the projection $\bbZ[H] \oplus \bbZ[H] \to I$ does not split in $\Perm(G)$.
\item[(ii)] The category of (integral) Artin motives $\AM(k)$ does {\em not} satisfy cancellation\footnote{It seems reasonable to conjecture that the above categories \eqref{eq:5-categories} also do {\em not} satisfy cancellation. Surprisingly, it appears that this problem (in the particular case of the category $\Perm(G)$) has not been studied; we have consulted a few experts on this matter.}!
\end{itemize}
\end{remark}
\begin{remark}
\label{rem:scot}
Let $H:=PSL(2, \bbF_{29})$.
L. Scott constructed in \cite{Scott} non-conjugate subgroups $L_1, L_2 \subset H$ such that $\bbZ[H/L_1]\simeq \bbZ[H/L_2]$ as $\bbZ[H]$-modules. Let us choose a Galois field extension $l/k$ inside $k_{\mathrm{sep}}$ with Galois group $H$ and write $l_i$ for the associated field extension $l^{L_i}$. Since $L_1$ is non-conjugate to $L_2$, $l_1\not \simeq l_2$. On the other, since the above equivalence of categories \eqref{eq:equivalence-MP} sends $l_i$  to the permutation $G$-module $\bbZ[H/L_i]$, we have $U(l_1) \simeq U(l_2)$.
\end{remark}
We finish this subsection with a result concerning inseparable field extensions:
\begin{theorem}\label{thm:inseparable}
Given a purely inseparable field extension $l/k$ of degree $p^r$, we have an isomorphism $U(k)_R \simeq U(l)_R$ for every commutative ring $R$ containing $1/p$.
\end{theorem}
Intuitively speaking, Theorem~\ref{thm:inseparable} shows that the noncommutative motives of  purely inseparable field extensions only contain torsion information.
\begin{corollary}\label{cor:inseparable}
Let $l/k$ be a purely inseparable field extension of degree $p^r$ and $E:\dgcat(k) \to \mathrm{D}$ an additive invariant with values in a $\bbZ[1/p]$-linear category. Under these assumptions, we have an isomorphism $E(k)\simeq E(l)$.
\end{corollary}
\begin{proof}
It follows from the combination of Theorem \ref{thm:inseparable} with equivalence \eqref{eq:categories-R}.
\end{proof}
\subsection*{Separable algebras}
Recall from \cite[\S III Thm.~3.1]{Knus} that a $k$-algebra $A$ is
separable if and only if it is isomorphic to a finite product of
matrix algebras $M_{r\times r}(D)$. Here, $D$ is a finite dimensional
division $k$-algebra with center $Z(D)$ a finite separable field
extension of $k$. Let us denote by $\Sep(k)$ the full subcategory of
$\NChow(k)$ consisting of the objects $U(A)$ with $A$ a separable
$k$-algebra. As explained above, $\Sep(k)$ is an additive symmetric
monoidal subcategory of $\NChow(k)$. The following result generalizes Proposition \ref{prop:properties}.
\begin{proposition}\label{prop:properties2}
The quotient functor $ \Sep(k) \to \NNum(k)$ is fully-faithful.
\end{proposition}
\begin{proof}
Since the category $\Sep(k)$ is rigid symmetric monoidal, it suffices to show that the composition bilinear pairings (see \S\ref{sub:universal})
$$
\Hom_{\Sep(k)}(U(k),U(A)) \times \Hom_{\Sep(k)}(U(A),U(k)) \too \Hom_{\Sep(k)}(U(k),U(k))
$$
are non-degenerate; see \cite[Lem.~7.1.1]{AK}. As mentioned above,
$U(\cA \times \cB) \simeq U(\cA) \oplus U(\cB)$ for all dg categories
$\cA$ and $\cB$. Moreover, thanks to Morita invariance (see
\S\ref{sec:dg} and \S\ref{sub:additive}), we have $U(M_{r\times
  r}(D))\simeq U(D)$. Hence, it suffices to treat the particular case
where $A$ is a division $k$-algebra $D$. In this case, since every
finitely generated projective right $D$-module is free and $K_0(D^\op)
\simeq K_0(D)$, the above pairing reduces to $\bbZ \times \bbZ \to
\bbZ, (m,n) \mapsto m \cdot \mathrm{dim}(D) \cdot n$. This pairing is clearly non-degenerate.
\end{proof}
\begin{notation}\label{not:new}
Given a finite $G$-set $S$, let us write $k_S$ for the commutative separable $k$-algebra $\Hom_G(S,k_{\mathrm{sep}})$. Note that $k_{S_1} \otimes k_{S_2} \simeq k_{S_1 \times S_2}$. In the same vein, given $s \in S$, let $k_s:= \Hom_G(Gs,k_{\mathrm{sep}})$. Note that $k_s=k_{\mathrm{sep}}^H$, where $H$ is the stabilizer of $s$. Finally, given an Azumaya algebra $A$ over $k_S$ (see \cite[\S III Thm.~5.1]{Knus}), let us write $A_s$ for the central simple $k_s$-algebra $A \otimes_{k_S}k_s$ and $\mathrm{ind}_s(A)$ for the index of~$A_s$.
\end{notation}
Consider now the following two categories:
\begin{itemize}
\item[(i)] Given finite $G$-sets $S_1,S_2$ and Azumaya algebras $A,B$ over $k_{S_1}$ and  $k_{S_2}$, respectively, let $\Map^{G, A, B}(S_1 \times S_2, \bbZ)$ be the subset of $\Map^G(S_1 \times S_2, \bbZ)$ consisting of those $G$-invariant functions $\alpha: S_1 \times S_2 \to \bbZ$ such that $\alpha((s_1,s_2)) \in \mathrm{ind}_{(s_1, s_2)} (A^\op \otimes B) \cdot \bbZ$ for every $(s_1, s_2) \in S_1 \times S_2$. The category $\Cov'(G)$ is defined as follows: the objects are the pairs $(S,A)$ with $S$ is a finite $G$-set and $A$ an Azumaya algebra over $k_S$; the morphisms $\Hom_{\Cov'(G)}((S_1,A),(S_2,B))$ are the functions $\Map^{G,A,B}(S_1\times S_2,\bbZ)$; the composition law and the identities are the same as those of the convolution category. Finally, the definitions
\begin{eqnarray*}
&(S_1,A) \oplus (S_2,B) := (S_1 \amalg S_2, A\times B) & (S_1,A) \otimes (S_2,B) := (S_1 \times S_2, A\otimes B)
\end{eqnarray*}
endow $\Cov'(G)$ with an additive symmetric monoidal structure.
\item[(ii)] Let $H, K \subseteq G$ be closed subgroups of finite index and $A,B$ finite dimensional division $k$-algebras with centers $k_{\mathrm{sep}}^H$ and $k_{\mathrm{sep}}^K$, respectively. Galois theory gives rise to the following isomorphism
\begin{eqnarray*}
k_{\mathrm{sep}}^H \otimes k_{\mathrm{sep}}^K \simeq \prod_{\overline{g} \in H \backslash G / K} l_{\overline{g}} &\mathrm{where} &
l_{\bar{g}}:=k_{\mathrm{sep}}^{g^{-1}Hg\cap K}\,.
\end{eqnarray*}
Making use of it, we conclude that
\begin{eqnarray}
 A^\op \otimes B &\simeq& A^{\op}\otimes_{k^H_{\mathrm{sep}}} (k^H_{\mathrm{sep}}\otimes k_{\mathrm{sep}}^K )
\otimes_{k_{\mathrm{sep}}^K} B \nonumber \\
&\simeq& \prod_{\overline{g} \in H \backslash G / K} A^{\op}\otimes_{k^H_{\mathrm{sep}}} l_{\overline{g}} \otimes_{k_{\mathrm{sep}}^K} B \nonumber \\
&\simeq &\prod_{\overline{g} \in H \backslash G / K} (l_{\overline{g}}\otimes_{k^H_{\mathrm{sep}}} A)^\op\otimes_{l_{\overline{g}}} (l_{\overline{g}}\otimes_{k^K_{\mathrm{sep}}} B)\,. \label{eq:product}
\end{eqnarray}
The $\overline{g}$-factor of \eqref{eq:product} is a central simple $l_{\overline{g}}$-algebra. Thanks to the Wedderburn theorem, it can be written as $ M_{r_{\overline{g}}\times r_{\overline{g}}}(D_{\overline{g}})$ for a unique integer $r_{\overline{g}} \geq 1$ and finite dimensional division $k$-algebra $D_{\overline{g}}$ with center $l_{\overline{g}}$. Let
$\Map^{A,B}(H \backslash G/ K, \bbZ)$ be the subset of $\Map(H
\backslash G/K, \bbZ)$ consisting of those functions $\alpha: H \backslash
G / K \to \bbZ$ such that $\alpha(\overline{g}) \in
\mathrm{ind}_{l_{\overline{g}}}(D_{\overline{g}}) \cdot \bbZ$ for every
$\overline{g} \in H \backslash G /K$. The category $\mathrm{Heck}'(G)$
is defined as follows: the objects are the pairs $(H,A)$ with $H
\subseteq G$ a closed subgroup of finite index and $A$ a central simple
$k_{\mathrm{sep}}^H$-algebra; the morphisms
$\Hom_{\Heck'(G)}((H,A),(K,B))$ are the functions $\Map^{A,B}(H \backslash
G / K, \bbZ)$; the composition law and the identities are the same as
those of the category $\Heck(G)$. The closure of $\mathrm{Heck}'(G)$ under
(formal) finite direct sums will be denoted by $\Hecke'(G)$. By construction, $\Hecke'(G)$ is an additive category.
\end{itemize}
The next result extends the equivalences $\CSep(k) \simeq \Cov(G)\simeq\Hecke(G)$ of Proposition \ref{thm:main1} to possibly noncommutative separable $k$-algebras.
\begin{theorem}\label{thm:main2}
\begin{itemize}
\item[(i)] The above data gives rise to well-defined additive categories $\Cov'(k)$ and $\Hecke'(G)$, with $\Cov'(k)$ being moreover symmetric monoidal.
\item[(ii)] The additive categories $\Sep(k), \Cov'(G)$ and $\Hecke'(G)$ are equivalent.
\end{itemize}
\end{theorem}
In what follows, we will equip $\Hecke'(G)$ with the symmetric monoidal structure inherited from 
$\Cov'(G)$ under the equivalence of Theorem \ref{thm:main2}.
\begin{corollary}\label{cor:center}
The assignment $A \mapsto Z(A)$ (where $Z(A)$ denotes the center of $A$) can be extended to an additive symmetric monoidal functor $ Z: \Sep(k) \to \CSep(k)$. This functor is moreover a retraction of the inclusion $\CSep(k) \subset \Sep(k)$.
\end{corollary}
\begin{proof}
Note that we have the canonical forgetful functor 
\begin{eqnarray}\label{eq:functor-1}
\Cov'(G) \too \Cov(G) && (S,A) \mapsto S
\end{eqnarray}
as well as the inclusion of categories
\begin{eqnarray}\label{eq:functor-2}
\Cov(G) \too \Cov'(G) && S\mapsto (S, k_S)\,.
\end{eqnarray}
Under the equivalences $\CSep(k) \simeq \Cov(G)$ and $\Sep(k)\simeq \Cov'(G)$, \eqref{eq:functor-1} (resp. \eqref{eq:functor-2}) identifies with the assignment $U(A) \mapsto U(Z(A))$ on objects (resp. with the inclusion of categories $\CSep(k) \subset \Sep(k)$). Therefore, the claim follows from the fact that the functors \eqref{eq:functor-1}-\eqref{eq:functor-2} are additive symmetric monoidal and from the equality $\eqref{eq:functor-1} \circ \eqref{eq:functor-2}=\Id$. 
\end{proof}
\subsection*{Central simple algebras}
Let $\CSA(k)$ be the full subcategory of $\Sep(k)$ consisting of the objects $U(A)$ with $A$ a central simple $k$-algebra, and $\CSA(k)^\oplus$ its closure under finite direct sums. Note that $\CSA(k)^\oplus$ is an additive symmetric monoidal subcategory of $\Sep(k)$. Moreover, by unraveling the above definitions, it is easy to see that we have the following 2-cartesian square of categories:
\begin{equation}\label{eq:diagram}
\xymatrix{
\{U(k)^{\oplus n}\,|\, n \geq 0\} \ar[d] \ar[r]  \ar@{}[dr]|{\ulcorner} & \CSA(k)^\oplus \ar[d] \\
\CSep(k) \ar[r] & \Sep(k)\,.
}
\end{equation}
Intuitively speaking, \eqref{eq:diagram} shows that the categories $\CSep(k)$ and $\CSA(k)^\oplus$ are ``orthogonal'', encoding respectively the commutative and the noncommutative
information. As a consequence of \cite[Thm.~2.1]{TV}, we have $U(l)_\bbQ \simeq
U(B)_\bbQ$ for every finite separable field extension $l/k$ and
central simple $l$-algebra $B$. Therefore, we obtain the following equivalences of categories:
\begin{eqnarray*}\label{eq:rational}
\{U(k)^{\oplus n}_\bbQ\,|\, n \geq 0\} \simeq \CSA(k)_\bbQ^\oplus && \CSep(k)_\bbQ \simeq \Sep(k)_\bbQ\,.
\end{eqnarray*}
Roughly speaking, the rational noncommutative information disappears!
\begin{remark}[Dg Azumaya algebras]
The classical notion of Azumaya algebra can be generalized to the differential graded setting; see Appendix \ref{app:dgAzumaya}. In {\em loc. cit.} we establish some properties of these dg Azumaya algebras and compute their noncommutative motives. In particular, we show that in this generality the rational noncommutative information does {\em not}~disappear. 
\end{remark}

Recall from \cite[Prop.~4.1.16]{Gille} that every central simple $k$-algebra $A$ admits a $p$-primary decomposition $A= \otimes_{p \in \cP} A^p$. Here, $\cP$ stands for the prime numbers and $A^p$ is characterized by the fact that its index is the $p$-primary component of $\mathrm{ind}(A)$.

As proved in \cite[Thm.~9.1]{twisted}, the following equivalence holds
\begin{equation}\label{eq:Brauer}
U(A) \simeq U(B) \Leftrightarrow [A]=[B]
\end{equation}
for any two central simple $k$-algebras $A$ and $B$. The following result extends the above equivalence \eqref{eq:Brauer} to a complete dictionary between objects of the category $\CSA(k)^\oplus$ and sequences of elements in the Brauer group $\mathrm{Br}(k)$.
\begin{theorem}
\label{thm:comprehensive}
The following holds:
\begin{enumerate}
\item[(i)] The category $\CSA(k)^{\oplus}$ is idempotent complete.
\item[(ii)] The indecomposable objects  in $\CSA(k)^{\oplus}$ are of the form $U(B)$.
\item[(iii)] Given central simple $k$-algebras $A_1, \ldots, A_n$, the indecomposable direct summands of $U(A_1) \oplus \cdots \oplus U(A_n)$ are of the form $U(B)$ with $B$ a central simple $k$-algebra satisfying the following condition: for every $p \in \cP$ there exists an integer $\varrho_p \in \{1, \ldots, n\}$ such that $[B^p]=[A^p_{\varrho_p}]$. Consequently, the Brauer class $[B]$ belongs to the subgroup of $\mathrm{Br}(k)$ generated by $\{[A_1], \ldots, [A_n]\}$. 
\item[(iv)]
Given central simple $k$-algebras $A_1, \ldots, A_n$ and $B_1, \ldots, B_m$, the following two conditions (a)-(b) are equivalent:
\begin{itemize}
\item[(a)] We have an isomorphism of noncommutative motives:
\begin{equation}\label{eq:sums}
U(A_1) \oplus \cdots \oplus U(A_n) \simeq U(B_1) \oplus \cdots \oplus U(B_m)\,.
\end{equation}
\item[(b)] The equality $n=m$ holds and for every $p\in \cP$ there exists a permutation $\sigma_p$ (which depends on $p$) such that $[B_i^p]=[A^p_{\sigma_p(i)}]$ for every $1 \leq i \leq n$.
\end{itemize}
\end{enumerate}
\end{theorem}
\begin{corollary}\label{cor:generate-Brauer}
  The above isomorphism \eqref{eq:sums} implies that the Brauer
  classes $[A_1], \ldots, [A_n]$ and $[B_1], \ldots, [B_n]$ generate
  the same subgroup of $\mathrm{Br}(k)$. Consequently, we obtain a
  well-defined map
$ \mathrm{Iso}(\CSA(k)^\oplus )\to
  \{\mathrm{subgroups}\,\,\mathrm{of}\,\,\mathrm{Br}(k)\}$.
\end{corollary}
\begin{proof}
  Recall from \cite[Prop.~4.5.13]{Gille} that the index and the period
  of a central simple $k$-algebra have the same prime factors. Hence, the
  proof follows from the combination of Theorem
  \ref{thm:comprehensive}(iv)(b) with the $p$-primary decomposition
  $\mathrm{Br}(k) = \prod_{p \in \cP} \mathrm{Br}(k)\{p\}$ of the
  Brauer group; see \cite[Prop.~4.5.16]{Gille}.
\end{proof}
\begin{corollary}\label{cor:cancellation}
Given central simple $k$-algebras $A_1, \ldots, A_n, B_1, \ldots, B_n, C$, we have the following cancellation property:
$$ \oplus_{i=1}^n U(A_i) \oplus U(C) \simeq \oplus_{i=1}^n U(B_i) \oplus U(C) \Rightarrow \oplus_{i=1}^n U(A_i) \simeq \oplus^n_{i=1} U(B_i)\,.$$
\end{corollary}
\begin{proof}
The proof follows from the fact that for every $p \in \cP$ the permutation $\sigma_p$ of item (iv)(b) of Theorem \ref{thm:comprehensive} can be choosen to fix the Brauer class $[C^p]$. 
\end{proof}
\begin{remark}
Thanks to Theorem \ref{thm:comprehensive} and Corollary \ref{cor:cancellation}, the additive category $\CSA(k)^\oplus$ has none of the pathologies described in Remark \ref{rem:swan}.
\end{remark}
We will prove Theorem \ref{thm:comprehensive} via a ring theoretic description of certain full subcategories of $\CSA(k)^\oplus$. Given central simple $k$-algebras $A_1, \ldots, A_n$, let us denote by $\mathrm{CSA}(A_1, \ldots, A_n)$ the full subcategory  of $\Sep(k)$ consisting of the objects $U(A_1), \ldots, U(A_n)$. Its closure under finite direct sums (resp. finite direct sums and direct factors) will be denoted by $\mathrm{CSA}(A_1, \ldots, A_n)^\oplus$ (resp. $\mathrm{CSA}(A_1, \ldots, A_n)^{\oplus,\natural}$). Consider the following ring $\Lambda(A_1, \ldots, A_n)$ of $n \times n$ matrices
\begin{eqnarray*}
\Lambda(A_1, \ldots, A_n)_{i,j}:= \mu_{ij}\bbZ && 1 \leq i,j \leq n \,,
\end{eqnarray*}
where $\mu_{ij}:=\mathrm{ind}(A_i^\op \otimes A_j)$. Since $\mu_{ij}=\mu_{ji}$ and $\mu_{ii}=1$, we have
\begin{equation}\label{eq:matrix}
\Lambda(A_1, \ldots, A_n)
=
\begin{pmatrix}
\bbZ & \mu_{12} \bbZ & \cdots & \mu_{1n} \bbZ\\
\mu_{12}\bbZ &\bbZ &\cdots &\mu_{2n}\bbZ\\
\vdots&\vdots&\ddots&\vdots\\
\mu_{1n}\bbZ&\mu_{2n}\bbZ&\cdots &\bbZ 
\end{pmatrix}_{n \times n}\,.
\end{equation}
\begin{proposition}\label{prop:ring}
The following holds:
\begin{itemize}
\item[(i)] We have a ring isomorphism $\mathrm{End}(\oplus_iU(A_i))\simeq \Lambda(A_1, \ldots, A_n)$.
\item[(ii)] We have an additive equivalence\footnote{This equivalence holds also when $n$ is infinite, provided we replace the ring $\Lambda(A_1, \ldots, A_n)$ by a category. We leave these straightforward generalizations to the reader.} of categories
\begin{equation*}\label{eq:CSA}
\varphi:\mathrm{CSA}(A_1, \ldots, A_n)^{\oplus, \natural} \simeq\mathrm{Proj}(\Lambda(A_1, \ldots, A_n)) \,\,\,\,\,\,\,\,\, U(A) \mapsto \Hom(\oplus_i U(A_i), U(A))\,,
\end{equation*}
where $\mathrm{Proj}$ stands for the category of finitely generated projective right modules.
\end{itemize}
\end{proposition}
\begin{remark}
The ring $\Lambda(A_1,\ldots,A_n)$ is a so-called {\em $\bbZ$-order}, \ie it is free of finite type as a $\bbZ$-module and the quotient ring is a central simple $\bbQ$-algebra. This class of rings plays a central role in integral representation theory; see Curtis-Reiner \cite{CR} for instance.  We were quite 
intrigued by the above connection
  between the classical theory of $\bbZ$-orders and the recent theory of noncommutative motives.
\end{remark}
Assume now that the Brauer classes $[A_1], \ldots, [A_n]$ form a subgroup $H$ of $\mathrm{Br}(k)$. This implies that the category $\CSA(A_1, \ldots, A_n)$ is symmetric monoidal and that $\CSA(A_1, \ldots, A_n)^\oplus$ and $\CSA(A_1, \ldots, A_n)^{\oplus, \natural}$ are additive symmetric monoidal.
\begin{proposition}\label{prop:groupcase}
Under the above assumption, the indecomposable objects in $\CSA(A_1, \ldots, A_n)^{\oplus,\natural}$ are of the form $U(A_i)$. Consequently, $\CSA(A_1, \ldots, A_n)^\oplus$ is idempotent complete.
\end{proposition}
Under the above assumption, we can give an alternative description of the category $\CSA(A_1,\ldots,A_n)^{\oplus}$ which is compatible with the symmetric monoidal structure. Consider the following $H$-graded subring of $\bbZ H$
\begin{equation}\label{eq:graded-ring}
\Sigma(A_1,\ldots,A_n):=\bigoplus_{i=1}^n \mu_i\bbZ[A_i]\text{$\,\,\subset \bbZ H$}\,,
\end{equation}
where $\mu_i:=\mathrm{ind}(A_i)$. Note that \eqref{eq:graded-ring} is indeed a subring of $\bbZ H$ since $\ind(A_i\otimes A_j)$ divides $\ind(A_i)\cdot \ind(A_j)$.
Given any $H$-graded commutative ring $\Sigma$ and $h \in H$, we will write $\Sigma(h)$ for the graded projective $\Sigma$-bimodule $\Sigma(h)_g:=\Sigma_{hg}$.
\begin{proposition}\label{prop:graded-ring}
  Assume that $[A_1],\ldots,[A_n]$ is a subgroup of the Brauer
  group. Then, we have an additive \emph{symmetric monoidal}
  equivalence\footnote{This equivalence holds also when $n$ is
    infinite.} of categories
\begin{eqnarray*}\label{eq:CSA1}
\mathrm{CSA}(A_1, \ldots, A_n)^\oplus \simeq \mathrm{Proj}_{\mathrm{gr}}(\Sigma(A_1,\ldots,A_n)) && U(A_i)\mapsto \Sigma(A_1,\ldots,A_n)([A_i]) \,.
\end{eqnarray*}
\end{proposition}
Let $A$ be a central simple $k$-algebra. As explained above, the (idempotent complete) category $\CSA(k, A, A^{\otimes 2}, \ldots,  A^{\otimes (\mathrm{per}(A)-1)})^\oplus$ is additive symmetric monoidal. Hence, we can consider the associated Grothendieck ring 
\begin{equation}\label{eq:Grothendieck}
K_0(\mathrm{CSA}(k, A, A^{\otimes 2}, \ldots,  A^{\otimes (\mathrm{per}(A)-1)})^\oplus)\,.
\end{equation}
\begin{theorem}\label{thm:computation}
Let $\prod_i p_i^{r_i}$ be the $p$-primary decomposition of $\mathrm{per}(A)$. Under these notations, the assignment $t \mapsto U(A)$ gives rise to a ring isomorphism
\begin{equation*}
\bbZ[t]/\langle(1-t) \prod_i (1+t+\cdots + t^{p_i^{r_i}-1})\rangle \simeq \eqref{eq:Grothendieck}\,.
\end{equation*}
\end{theorem}
Thanks to Theorem \ref{thm:computation}, the rank of \eqref{eq:Grothendieck} is equal to $\sum_i (p_i^{r_i}-1) +1$. Note that this number is strictly inferior to the period $\mathrm{per}(A) = \prod_i p_i^{r_i}$ whenever there are at least two distinct prime numbers.
\section{Applications}\label{sec:applications}
\subsection*{Motivic relations}
Let $A$ be a central simple $k$-algebra, $\prod_i p_i^{r_i}$ the $p$-primary decomposition of $\mathrm{per}(A)$, and $p_A(t)$ the polynomial $ \prod_i (1 +t + \cdots + t^{p_i^{r_i}-1})$.
\begin{notation}
Given a polynomial $p(t)=a_0 + a_1 t + a_2 t^2 + \cdots + a_n t^n \in \bbN[t]$ and an additive invariant $E: \dgcat(k) \to \mathrm{D}$, let $E(A;p(t))$ be the following direct sum
\begin{equation}\label{eq:sum-E}
E(k)^{\oplus a_0} \oplus E(A)^{\oplus a_1} \oplus E(A^{\otimes 2})^{\otimes a_2} \oplus \cdots \oplus E(A^{\otimes n})^{\oplus a_n} \in \mathrm{D} \,.
\end{equation}
\end{notation}
Our first family of motivic relations is the following:
\begin{proposition}\label{prop:relations1}
Under the above notations, we have $E(A;p_A(t))\simeq E(A;tp_A(t))$. 
\end{proposition}
\begin{proof}
Consider the noncommutative motives $U(A;p_A(t))$ and $U(A;tp_A(t))$. Theorem \ref{thm:computation} implies that their classes in the Grothendieck ring \eqref{eq:Grothendieck} are the same. Hence, making use of the cancellation property of Corollary \ref{cor:cancellation}, we conclude that their are isomorphic; see \cite[II \S2 Cor.~1.2]{Weibel}. The proof follows now from the equivalence of categories \eqref{eq:categories}.
\end{proof}
\begin{example}\label{ex:simple}
Let $A$ be the reader's favorite central simple $k$-algebra with $\mathrm{per}(A)=6$. In this particular case, $p_A(t)=(1+t)(1+t +t^2)$. By combining Proposition \ref{prop:relations1} with Corollary \ref{cor:cancellation}, we hence obtain $E(k) \oplus E(A)\simeq E(A^{\otimes 3}) \oplus E(A^{\otimes 4})$.
\end{example}
Our second family of motivic relations, which greatly generalizes the preceding Example \ref{ex:simple}, is the following:
\begin{proposition}\label{prop:relations2}
Let $A, B, C$ be central simple $k$-algebras and $E:\dgcat(k) \to \mathrm{D}$ an additive invariant. Assume that $\mathrm{ind}(A)$ and $\mathrm{ind}(B)$ are coprime. Under these assumptions, we have an isomorphism
\begin{equation}\label{eq:relations2}
E(C) \oplus E(A\otimes B\otimes C) \simeq E(A\otimes C) \oplus E(B\otimes C)\,.
\end{equation}
\end{proposition}
\begin{proof}
  Let $\prod_i p_i^{r_i}$ (resp. $\prod_j q_j^{s_j}$) be the
    $p$-primary decomposition of $\mathrm{ind}(A)$
    (resp. $\mathrm{ind}(B)$). Since $\mathrm{ind}(A)$ and
    $\mathrm{ind}(B)$ are coprime, $p_i \neq q_j$ for every $i$ and
    $j$. Consequently, we have $A^{q_j}=k$ and $B^{p_i}=k$. This
    implies that $(A\otimes B\otimes C)^{q_j}= B^{q_j}\otimes C^{q_j}$
    and $(A\otimes B\otimes C)^{p_i}=A^{p_i}\otimes C^{p_i}$. Making
    use of Theorem \ref{thm:comprehensive}(iv)(b), we hence obtain an
    isomorphism $U(C) \oplus U(A\otimes B\otimes C) \simeq U(A\otimes
    C) \oplus U(B\otimes C)$. The proof follows now from the
    equivalence of categories \eqref{eq:categories}.
\end{proof}
\begin{example}
Wedderburn theorem implies that $A\simeq M_{r \times r}(D_A), B\simeq M_{s\times s}(D_B)$, and $A \otimes B\simeq M_{t \times t}(D_{A\otimes B})$, for unique integers $r,s,t \geq 1$ and division $k$-algebras $D_A,D_B$ and $D_{A\otimes B}$. Therefore, the above isomorphism \eqref{eq:relations2}, with $C=k$ and $E$ equal to the first algebraic $K$-theory group (see \cite[\S2.8]{Gille}), reduces to
\begin{equation}\label{eq:relations-K_1}
k^\times \times D_{A\otimes B}^\times/[D_{A\otimes B}^\times, D_{A\otimes B}^\times] \simeq D_A^\times/[D_A^\times, D_A^\times] \times D_B^\times/[D_B^\times, D_B^\times]\,.
\end{equation}
\end{example}
To the best of the authors knowledge, all the above relations \eqref{eq:relations2} (in particular \eqref{eq:relations-K_1}) are new in the literature. More generally, given central simple $k$-algebras $A_1, \ldots, A_n, B_1, \ldots, B_n$ satisfying Theorem \ref{thm:comprehensive}(iv)(b), we have an isomorphism 
\begin{equation}\label{eq:more-relations}
E(A_1) \oplus \cdots \oplus E(A_n) \simeq E(B_1) \oplus \cdots \oplus E(B_n)\,.
\end{equation}
These latter relations can be nevertheless recovered from the previous ones:
\begin{proposition}\label{prop:implication}
\begin{itemize}
\item[(i)] The above relations \eqref{eq:relations2} imply all the relations \eqref{eq:more-relations}.
\item[(ii)] When $E$ is moreover symmetric monoidal, the above relations \eqref{eq:relations2} with $C=k$ imply all the relations \eqref{eq:more-relations}.
\end{itemize}
\end{proposition}
\subsection*{Severi-Brauer varieties}
Let $A$ be a central simple $k$-algebra and $\mathrm{SB}(A)$ the associated Severi-Brauer variety. As proved in \cite[Prop.~2.8]{MT}, we have an isomorphism
\begin{equation}\label{eq:decomp-SB}
U(\perf_\dg(\mathrm{SB}(A))) \simeq U(k) \oplus U(A) \oplus U(A)^{\otimes 2} \oplus \cdots \oplus U(A)^{\otimes (\mathrm{deg}(A)-1)}\,.
\end{equation}
As an application of the above Theorem~\ref{thm:comprehensive}(iv), we obtain the following result:
\begin{theorem}\label{thm:SB}
Given central simple $k$-algebras $A$ and $B$ with the same degree, the following three conditions are equivalent:
\begin{itemize}
\item[(i)] We have an isomorphism $U(\perf_\dg(\mathrm{SB}(A))) \simeq U(\perf_\dg(\mathrm{SB}(B)))$.
\item[(ii)] The Brauer classes $[A]$ and $[B]$ generate the same subgroup of $\mathrm{Br}(k)$.
\item[(iii)] There exists an integer $i$, coprime to $\mathrm{per}(A)$, such that $[B]=[A]^i$.
\end{itemize}
\end{theorem}
\begin{remark}
N. Karpenko proved in \cite[Criterion~7.1]{Karpenko} that we have an isomorphism between Chow motives $M(\mathrm{SB}(A))\simeq M(\mathrm{SB}(B))$ if and only if $A$ is isomorphic to $B$ or to $B^\op$. Roughly speaking, this shows that conditions (ii) $\Leftrightarrow$ (iii) of Theorem \ref{thm:SB} are truly a noncommutative phenomenon.
\end{remark}
\begin{corollary}\label{cor:Severi}
Let $A$ be a central simple $k$-algebra, $i$ an integer coprime to $\mathrm{per}(A)$, and $E:\dgcat(k) \to \mathrm{D}$ an additive invariant. Under these assumptions, we have an isomorphism $E(\mathrm{SB}(A)) \simeq E(\mathrm{SB}(A^{\otimes i}))$. 
\end{corollary}
The following result relates noncommutative motives with birationality:
\begin{proposition}\label{prop:birational}
The universal additive invariant $U$ (and therefore every additive invariant) is a birational invariant of Severi-Brauer varieties.
\end{proposition}
\begin{proof}
As proved by Amitsur in \cite[\S9]{Amitsur}, if two Severi-Brauer varieties $\mathrm{SB}(A)$ and $\mathrm{SB}(B)$ are birational, then $A$ and $B$ have the same degree and $[A]$ and $[B]$ generate the same subgroup of $\mathrm{Br}(k)$. The proof follows then from Theorem \ref{thm:SB}.
\end{proof}
\begin{remark}[Amitsur's conjecture]
S. Amitsur conjectured in \cite{Amitsur} that if two central simple $k$-algebras $A$ and $B$ have the same degree and $[A]$ and $[B]$ generate the same subgroup of $\mathrm{Br}(k)$, then $\mathrm{SB}(A)$ and $\mathrm{SB}(B)$ are birational. Thanks to Proposition \ref{prop:birational}, Amitsur's conjecture can be reformulated as follows:

\smallskip

{\it Conjecture: $U$ is a complete birational invariant of Severi-Brauer varieties.}
\end{remark}
\subsection*{Twisted flag varieties}
Let $A$ be a central simple $k$-algebra, $d_1, \ldots, d_m$, $m \geq 1$, positive
integers such that 
$\sum_i d_i=\mathrm{deg}(A)$, and $\mathrm{Flag}(d_1, \ldots, d_m; A)$ the associated twisted flag
variety. Recall that when $m=1$, $\mathrm{Flag}(d_1; A)$ identifies with
the twisted Grassmanian variety $\mathrm{Gr}(d_1;A)$ and when $d_1=1$,
$\mathrm{Gr}(d_1;A)$ reduces to $\mathrm{SB}(A)$.
\begin{notation}\label{not:notation2}
Given an additive symmetric monoidal category $(\cC,\otimes, {\bf 1})$, an object $b \in \cC$, and a polynomial $p(t)=a_0 + a_1 t + a_2 t^2+ \cdots + a_nt^n \in \bbN[t]$, let $p(b)$ be the direct sum
$ {\bf 1}^{\oplus a_0} \oplus b^{\oplus a_1} \oplus (b^{\otimes 2})^{\oplus a_2} \oplus \cdots \oplus (b^{\otimes n})^{\oplus a_n} \in \cC$. Note that the above direct sum \eqref{eq:sum-E} reduces to $p(E(A))$ when $E$ is moreover symmetric monoidal.
\end{notation}
\begin{theorem}\label{thm:Flag}
We have the following motivic decomposition 
\begin{equation}\label{eq:motivic-decomp-flag}
U(\perf_\dg(\mathrm{Flag}(d_1, \ldots, d_m;A)))\simeq {\mathrm{deg}(A) \choose d_1 \cdots d_m}_{U(A)}\,,
\end{equation}
where ${\mathrm{deg}(A) \choose d_1 \cdots d_m}_t$ stands for the Gaussian polynomial; see Example \ref{ex:cyclic}.
\end{theorem}
Note that \eqref{eq:motivic-decomp-flag} generalizes the above motivic decomposition \eqref{eq:decomp-SB}. In the same vein, the following results 
generalize Theorem~\ref{thm:SB} and Corollary~\ref{cor:Severi}.
\begin{theorem}\label{thm:computation2}
Given central simple $k$-algebras $A$ and $B$ with the same degree, the following three conditions are equivalent:
\begin{itemize}
\item[(i)] We have an isomorphism of noncommutative motives
$$U(\perf_\dg(\mathrm{Flag}(d_1, \ldots, d_m;A))) \simeq U(\perf_\dg(\mathrm{Flag}(d_1, \ldots, d_m;B)))\,.$$
\item[(ii)] The Brauer classes $[A]$ and $[B]$ generate the same subgroup of $\mathrm{Br}(k)$.
\item[(iii)] There exists an integer $i$, coprime to $\mathrm{per}(A)$, such that $[B]=[A]^i$.
\end{itemize}
\end{theorem}
\begin{proof}
Assume condition (i). Thanks to isomorphism \eqref{eq:motivic-decomp-flag} and Corollary \ref{cor:generate-Brauer}, the Brauer classes $[k], [A], \ldots, [A^{\otimes(\mathrm{per}(A)-1)}]$ and $[k], [B], \ldots, [B^{\otimes(\mathrm{per}(B)-1)}]$ generate the same subgroup of $\mathrm{Br}(k)$. This implies condition (ii). Conditions (ii) and (iii) are clearly equivalent. Assume now condition (iii). Thanks to the equivalence \eqref{eq:Brauer}, we have $U(B) \simeq U(A)^{\otimes i}$. Hence, condition (i) follows from
\begin{eqnarray}
U(\perf_\dg(\mathrm{Flag}(d_1, \ldots, d_m;B))) & \simeq & {\mathrm{deg}(B) \choose d_1 \cdots d_m}_{U(B)} \nonumber \\
& \simeq & {\mathrm{deg}(A) \choose d_1 \cdots d_m}_{U(A)^{\otimes i}} \nonumber \\
\label{eq:cite-sieving}& \simeq &  {\mathrm{deg}(A) \choose d_1 \cdots d_m}_{U(A)} \label{eq:app-cyc-sieving} \\
&\simeq & U(\perf_\dg(\mathrm{Flag}(d_1, \ldots, d_m;A))) \,, \nonumber
\end{eqnarray}
where \eqref{eq:cite-sieving} follows from the cyclic sieving phenomenon; see Corollary \ref{cor:sieving}.
\end{proof}
\begin{corollary}\label{cor:Flag}
Let $A$ be a central simple $k$-algebra, $d_1, \ldots, d_m, m \geq 1$, positive integers such that $\sum_i d_i=\mathrm{deg}(A)$, $i$ an integer coprime to $\mathrm{per}(A)$, and $E:\dgcat(k) \to \mathrm{D}$ an additive invariant. Under these assumptions, we have
$$ E(\mathrm{Flag}(d_1,\ldots,d_m; A)) \simeq E(\mathrm{Flag}(d_1,\ldots,d_m; A^{\otimes i}))\,.$$
\end{corollary}
\begin{proof}
It follows from the combination of Theorem \ref{thm:computation2} with equivalence \eqref{eq:categories}.
\end{proof}  
\medbreak\noindent\textbf{Acknowledgments:}
The authors would like to thank Christine Bessenrodt and Robert Guralnick for their comments on some questions concerning integral representation theory, and also to Marcello Bernardara for bringing our attention to Amitsur's work \cite{Amitsur}.
\section{Background on dg categories}\label{sec:dg}
Let $\cC(k)$ be the symmetric monoidal category of cochain complexes of $k$-vector spaces. A {\em dg category $\cA$} is a category enriched over $\cC(k)$ and a {\em dg functor} $F:\cA\to \cB$ is a functor enriched over $\cC(k)$; consult Keller's ICM survey \cite{ICM-Keller} for further details. 

Let $\cA$ be a dg category. Its opposite dg category $\cA^\op$ has the same objects and $\cA^\op(x,y):=\cA(y,x)$.  A right $\cA$-module is a dg functor $\cA^\op \to \cC_\dg(k)$ with values in the dg category $\cC_\dg(k)$ of complexes of $k$-vector spaces. Let $\cC(\cA)$ be the category of right $\cA$-modules. Following \cite[\S3]{ICM-Keller}, the derived category $\cD(\cA)$ of $\cA$ is defined as the localization of $\cC(\cA)$ with respect to the class of objectwise quasi-isomorphisms. Its full subcategory of compact objects will be denoted by $\cD_c(\cA)$.

A dg functor $F:\cA\to \cB$ is called a {\em Morita equivalence} if it induces an equivalence $\cD(\cA) \stackrel{\simeq}{\to} \cD(\cB)$ on derived categories; see \cite[\S4.6]{ICM-Keller}. As proved in \cite[Thm.~5.3]{Additive}, $\dgcat(k)$ admits a Quillen model structure whose weak equivalences are the Morita equivalences. Let $\Hmo(k)$ be the homotopy category hence obtained. 

The  tensor product $\cA\otimes\cB$ of dg categories is defined as follows: the set of objects is the cartesian product and $(\cA\otimes\cB)((x,w),(y,z)):= \cA(x,y) \otimes \cB(w,z)$. As explained in \cite[\S2.3]{ICM-Keller}, this construction gives rise to symmetric monoidal categories $(\dgcat(k),-\otimes-,k)$ and $(\Hmo(k),-\otimes-,k)$. Finally, given dg categories $\cA$ and $\cB$, an  $\cA\text{-}\cB$-bimodule is a dg functor $\mathrm{B}:\cA \otimes \cB^\op\to \cC_\dg(k)$, \ie a right $(\cA^\op \otimes \cB)$-module. Associated to a dg functor $F: \cA \to \cB$, we have the $\cA\text{-}\cB$-bimodule 
\begin{eqnarray}\label{eq:bimodule2}
{}_F\cB:\cA\otimes \cB^\op \too \cC_\dg(k) && (x,w) \mapsto \cB(w,F(x))
\end{eqnarray}
as well as the $\cB\text{-}\cA$-bimodule
\begin{eqnarray}\label{eq:bimodule3}
\cB_F:\cB\otimes \cA^\op \too \cC_\dg(k) && (w,x) \mapsto \cB(F(x),w)\,.
\end{eqnarray}
\section{Background on noncommutative motives}\label{sec:NCmotives}
\subsection{Additive invariants}\label{sub:additive}
Given a dg category $\cA$, let $T(\cA)$ be the dg category of pairs $(i,x)$, where $ i \in \{1,2\}$ and $x\in\cA$. The complex of morphisms in $T(\cA)$ from $(i,x)$ to $(i',x')$ is given by $\cA(x,x')$ if $ i' \geq i$ and is zero otherwise. Composition is induced by $\cA$. Intuitively speaking, $T(\cA)$ ``dg categorifies'' the notion of upper triangular matrix. Note that we have two inclusion dg functors $i_1, i_2:\cA \hookrightarrow T(\cA)$. 

Recall from \cite{Additive} that a functor $E:\dgcat(k) \to \mathrm{D}$, with values in an additive category $\mathrm{D}$, is called an {\em additive invariant} if it satisfies the following conditions:
\begin{itemize}
\item[(i)] it sends Morita equivalences to isomorphisms;
\item[(ii)] given a dg category $\cA$, the inclusion dg functors $i_1, i_2$ induce an isomorphism\footnote{Condition (ii) can be equivalently formulated in terms of semi-orthogonal decompositions in the
sense of Bondal-Orlov~\cite{BO}; see \cite[Thm.~6.3(4)]{Additive}.}
$$E(\cA) \oplus E(\cA) \stackrel{\sim}{\too} E(T(\cA))\,.$$
\end{itemize}
Examples of additive invariants include algebraic $K$-theory, cyclic homology (and all its variants), topological Hochschild homology, etc; consult the survey \cite{Buenos}. When applied to the dg categories $\perf_\dg(X)$, these invariants agree with the corresponding invariants of the quasi-compact quasi-separated $k$-schemes $X$.
\subsection{Universal additive invariant}\label{sub:universal}
Given dg categories $\cA$ and $\cB$, let $\rep(\cA,\cB)$ be the full triangulated subcategory of $\cD(\cA^\op \otimes \cB)$ consisting of those
$\cA\text{-}\cB$-bimodules $\mathrm{B}$ such that for every $x \in \cA$ the associated right $\cB$-module $\mathrm{B}(x,-)$ belongs to
$\cD_c(\cB)$. As proved in \cite[Cor.~5.10]{Additive}, there is a bijection 
$\Hom_{\Hmo(k)}(\cA,\cB)\simeq \mathrm{Iso}\,\rep(\cA,\cB)$ under which the composition law of $\Hmo(k)$
corresponds to the (derived) tensor product of bimodules. Since the above 
$\cA\text{-}\cB$-bimodules \eqref{eq:bimodule2} clearly belong to
$\rep(\cA,\cB)$, we hence obtain a well-defined symmetric monoidal functor
\begin{eqnarray}\label{eq:functor1}
\dgcat(k) \too \Hmo(k) && F \mapsto {}_F\cB\,.
\end{eqnarray}
The {\em additivization} of $\Hmo(k)$ is the additive category $\Hmo_0(k)$ with the same objects as $\Hmo(k)$ and with morphisms given by $\Hom_{\Hmo_0(k)}(\cA,\cB):=K_0\rep(\cA,\cB)$, where $K_0\rep(\cA,\cB)$ stands for the Grothendieck group of the triangulated category $\rep(\cA,\cB)$. The composition law is induced by the (derived) tensor product of bimodules and the symmetric monoidal structure extends by bilinearity from $\Hmo(k)$ to $\Hmo_0(k)$. Note that we have also a well-defined symmetric monoidal functor
\begin{eqnarray}\label{eq:functor2}
\Hmo(k) \too \Hmo_0(k) && \mathrm{B} \mapsto [\mathrm{B}]\,.
\end{eqnarray}
As proved in \cite[Thms.~5.3 and 6.3]{Additive}, the universal additive invariant $U$ is obtained by composing the above functors \eqref{eq:functor1} and \eqref{eq:functor2}. Finally, recall from \cite[Lem.~4.2]{Additive} that $U(\cA) \oplus U(\cB) \simeq U(\cA \times \cB) \simeq U(\cA\amalg \cB)$ for any two dg categories $\cA$ and $\cB$.
\subsection{Coefficients}\label{sub:coefficients}
Let $R$ be a commutative ring. The {\em $R$-linearization} of $\Hmo_0(k)$ is the $R$-linear category $\Hmo_0(k)_R$ obtained by tensoring the morphisms of $\Hmo_0(k)$ with $R$. Note that $\Hmo_0(k)_R$ inherits a $R$-linear symmetric monoidal structure and that we have a well-defined symmetric monoidal functor
\begin{eqnarray}\label{eq:functor3}
\Hmo_0(k) \too \Hmo_0(k)_R && [\mathrm{B}] \mapsto [\mathrm{B}]\otimes_\bbZ R\,.
\end{eqnarray}
The {\em universal additive invariant with $R$-coefficients $U(-)_R$} is obtained by composing the universal additive invariant $U$ with \eqref{eq:functor3}. Given any $R$-linear additive category $\mathrm{D}$, there is an induced equivalence of categories
\begin{equation}\label{eq:categories-R}
U(-)_R^\ast: \Fun(\Hmo_0(k)_R,\mathrm{D}) \stackrel{\simeq}{\too} \Fun_{\mathrm{add}}(\dgcat(k),\mathrm{D})\,,
\end{equation}
where the left-hand side denotes the category of $R$-linear additive functors.
\section{Proof of Theorem \ref{thm:motivation}}
Let $\cA$ be a smooth proper dg category. Since $\NNum(k)$ is a rigid symmetric monoidal category, it suffices to show that $\Hom_{\NNum(k)}(U(k),U(\cA))$ is a finitely generated abelian group. Recall from \cite[\S4]{Kontsevich} that the Grothendieck group $K_0(\cA):= K_0(\cD_c(\cA))$ of $\cA$ comes equipped with the bilinear form:
\begin{eqnarray*}
\chi: K_0(\cA) \times K_0(\cA) \too \bbZ && (P,Q) \mapsto \sum_i (-1)^i \mathrm{dim}\, \Hom_{\cD_c(\cA)}(P,Q[i])\,.
\end{eqnarray*}
This form is in general not symmetric neither anti-symmetric. However, as proved in \cite[Thm.~4.3]{Kontsevich}, the left and right kernels agree; let $\mathrm{Ker}(\chi)$ be the resulting kernel. Note now that the proof of \cite[Thm.~1.1]{Kontsevich} (with $F$ replaced by $\bbZ$) implies that 
\begin{equation}\label{eq:description-key}
\Hom_{\NNum(k)}(U(k),U(\cA)) \simeq K_0(\cA)/_{\!\mathrm{Ker}(\chi)}\,.
\end{equation}
Hence, it is enough to show that the right-hand side of \eqref{eq:description-key} is finitely generated. As explained in \cite[\S4]{Galois}, the $\otimes$-ideal $\cN$ is compatible with extension of scalars. Consequently, we have isomorphisms of $\bbQ$-vector spaces:
$$ (K_0(\cA)/_{\!\mathrm{Ker}(\chi)})_\bbQ \simeq K_0(\cA)_\bbQ/_{\!\mathrm{Ker}(\chi_\bbQ)} \simeq \Hom_{\NNum(k)_\bbQ}(U(k)_\bbQ,U(\cA)_\bbQ)\,.$$
Since by assumption $k$ is of characteristic zero, the category $\NNum(k)_\bbQ$ is abelian semi-simple; see \cite[Thm.~1.10]{Semisimple}. This implies in particular that the above $\bbQ$-vector spaces are finite dimensional. Consider now the injective group homomorphism 
\begin{equation}\label{eq:injective}
K_0(\cA)/_{\!\mathrm{Ker}(\chi)} \too \Hom_\bbZ(K_0(\cA)/_{\!\mathrm{Ker}(\chi)}, \bbZ)
\end{equation}
induced by the above bilinear form $\chi$. Since the $\bbQ$-vector space $(K_0(\cA)/_{\!\mathrm{Ker}(\chi)})_\bbQ$ is finite dimensional, the right-hand side of \eqref{eq:injective} is finitely generated. Hence, we conclude finally that the abelian group $K_0(\cA)/_{\!\mathrm{Ker}(\chi)}$ is also finitely generated.
\section{Proof of Proposition \ref{thm:main1}}
\label{sec:proof:main1}
Consider the following composition
\begin{equation}\label{eq:comp-etale}
\Etale(k)\subseteq \Chow(k)\stackrel{\pi}{\too} \Chow(k)/_{\!-\otimes \bbZ(1)}\,,
\end{equation}
where $\bbZ(1)$ stands for the Tate motive and $\Chow(k)/_{\!-\otimes \bbZ(1)}$ for the orbit category; see \cite[\S7]{CvsNC}. As explained in the proof of \cite[Thm.~1.1]{Artin}, the functor \eqref{eq:comp-etale} is not only additive and symmetric monoidal but moreover fully-faithful. Recall now from \cite[Props.~4.1 and 6.6]{MT} the construction of the additive symmetric monoidal functor $\Psi$ making the following diagram commute
\begin{equation}\label{eq:square}
\xymatrix@C=3em@R=1.5em{
\mathrm{etale}(k)^\op \ar[d]_-{M} \ar@{=}[rrr] &&& \mathrm{etale}(k)^\op \ar[dd]^-{X \mapsto U(\perf_\dg(X))} \\
\Etale(k) \ar[d]_-{\eqref{eq:comp-etale}} &&&\\
\Chow(k)/_{\!-\otimes \bbZ(1)} &&& \CSep(k) \ar[lll]^-{\Psi} \,,
}
\end{equation}
where $\mathrm{etale}(k)$ stands for the category of finite {\'e}tale $k$-schemes. Since $\Psi$ is symmetric monoidal we conclude automatically from \eqref{eq:square} that $\Etale(k) \simeq \CSep(k)$.

Let us now prove that $\Cov(G) \simeq \Perm(G)$. Given finite $G$-sets $S_1$ and $S_2$, a simple verification shows that the assignments $S_1 \mapsto \bbZ[S_1]$ and 
\begin{eqnarray*}
\Map^G(S_1 \times S_2,\bbZ) \too \Hom_{\bbZ[G]}(\bbZ[S_1],\bbZ[S_2]) && \alpha \mapsto (s_1 \mapsto \sum_{s_2 \in S_2} \alpha(s_1,s_2) s_2)
\end{eqnarray*}
give rise to a functor $\Cov(G) \to \mathrm{Perm}(G)$. Thanks to the natural identifications
\begin{eqnarray*}
\bbZ[S] \stackrel{s \mapsto \delta_s}{\simeq} \Fun(S,\bbZ) && \bbZ[S_1]\oplus \bbZ[S_2] \simeq \bbZ[S_1 \amalg S_2]\,,
\end{eqnarray*}
this latter functor is moreover fully-faithful and additive. By definition of $\Perm(G)$, we conclude that it is furthermore essentially surjective and hence and equivalence.

Let us now prove that $\Hecke(G) \simeq \Cov(G)$. Given closed subgroups $H,K \subseteq G$ of finite index, we claim that the assignments $H \mapsto G/H$ and
\begin{equation}\label{eq:maps1}
\Map(H\backslash G /K, \bbZ) \too \Map^G(G/H\times G/K, \bbZ) \quad \alpha \mapsto \alpha'(\overline{g_1}, \overline{g_2}):= \alpha(g_1^{-1} g_2)
\end{equation}
give rise to a fully-faithful functor $\Heck(G) \to \Cov(G)$. Clearly, the identities are preserved. Given closed subgroups $H,K,L \subseteq G$ and functions $\alpha \in \Map(H\backslash G/K,\bbZ)$ and $\beta \in \Map(K\backslash G/L,\bbZ)$, the value of $\alpha' \ast \beta'$ (resp. $(\alpha \star \beta)'$) at $(\overline{g_1}, \overline{g_3})$ equals
\begin{eqnarray*}
\sum_{\overline{g_2} \in G/K} \alpha(g_1^{-1} g_2) \cdot \beta(g_2^{-1}g_3) && (\text{resp.} \sum_{\overline{h} \in G/K} \alpha(h^{-1}) \cdot \beta(hg_1^{-1}g_3))\,.
\end{eqnarray*}
Since the right-hand side can be obtained from the left-hand side via the substitution $\overline{g_2} \mapsto g_1 \overline{h^{-1}}$, we hence conclude that $\alpha' \ast \beta' = (\alpha \star \beta)'$, \ie that the composite law is also preserved. In what concerns fully-faithfulness, note that \eqref{eq:maps1} is an isomorphism; it inverse is given by $\gamma \mapsto \gamma(\overline{1},-)$. This implies our claim. Now, consider the (unique) fully-faithful functor $\Hecke(G) \to \Cov(G)$ which extends the above functor and preserves finite direct sums. Since every finite $G$-set $S$ decomposes into the disjoint union $\amalg_i G/H_i$ of its orbits, we conclude (using Galois theory) that the latter functor is moreover essentially surjective and hence an equivalence.

Finally, recall from \cite[\S4.1.6]{Andre} that the idempotent completion of $\Etale(k)$ identifies with the category of (integral) Artin motives $\AM(k)$. This achieves the proof.

\begin{remark}\label{rk:equivalence-G}
As the above proof shows, given an arbitrary group $G$, the categories $\Cov(G), \Hecke(G)$ and $\Perm(G)$ are equivalent.
\end{remark}

\begin{remark}
By composing $\Cov(G) \simeq \Perm(G)$ with the inverse of \eqref{eq:equivalence-MP}, we obtain the additive symmetric monoidal equivalence of categories (see Notation~\ref{not:new}):
\begin{eqnarray}\label{eq:equivalence-new}
\Cov(G) \stackrel{\simeq}{\too} \CSep(k) && S \mapsto U(k_S) \,.
\end{eqnarray}
Given finite $G$-sets $S_1$ and $S_2$, the induced isomorphism 
$$ \Map^G(S_1 \times S_2, \bbZ) \stackrel{\sim}{\too} \Hom_{\CSep(k)}(U(k_{S_1}),U(k_{S_2})) \simeq K_0(k_{S_1 \times S_2})$$
sends the characteristic function $\delta_{(s_1,s_2)}$ of the $G$-orbit of $(s_1,s_2) \in S_1 \times S_2$ to the class $[k_{(s_1,s_2)}]$ of the finitely generated projective right $k_{S_1 \times S_2}$-module $k_{(s_1,s_2)}$.
\end{remark}
\section{Proof of Theorem \ref{thm:inseparable}}
Note first that every ring homomorphism $R \to R'$ gives rise to a well-defined functor $\Hmo_0(k)_R \to \Hmo_0(k)_{R'}$. Hence, since $\bbZ[1/p]$ is initial among all the rings containing $1/p$, it suffices to prove the particular case $R=\bbZ[1/p]$.

Let us denote by $\iota: k \to l$ the field extension homomorphism. Recall from \eqref{eq:bimodule2}\text{-}\eqref{eq:bimodule3} that $\iota$ gives rise to a $k\text{-}l$-bimodule ${}_\iota l$ and also to a $l\text{-}k$-bimodule $l_\iota$. Note that since the extension $l/k$ is finite, $l_\iota \in \rep(l,k)$. Consider now the composition
\begin{equation}\label{eq:comp-1}
U(k) \stackrel{[{}_\iota l]}{\too} U(l) \stackrel{[l_\iota]}{\too} U(k)\,.
\end{equation}
By definition of the category $\Hmo_0(k)$, \eqref{eq:comp-1} identifies with the Grothendieck class $[l] \in K_0 \rep(k,k) \simeq K_0(k) \simeq \bbZ$. Hence, since $[l:k]=p^r$, we conclude that \eqref{eq:comp-1} is equal to $p^r \cdot \id_{U(k)}$. Consider now the other composition
\begin{equation}\label{eq:comp-2}
U(l)  \stackrel{[l_\iota]}{\too} U(k) \stackrel{[{}_\iota l]}{\too} U(l)\,.
\end{equation}
In this case, \eqref{eq:comp-2} identifies with the class $[l\otimes_k l] \in K_0\rep(l,l) \simeq K_0(l \otimes l)$. Since by hypothesis $l/k$ is a purely inseparable field extension, \cite[Lem.~9.6]{TV} implies that the $k$-algebra $l \otimes_k l$ is local. Hence, $K_0(l \otimes_k l) \simeq \bbZ$. Using the equality 
$$\frac{\mathrm{dim}_k(l \otimes_k l)}{\mathrm{dim}_k(l)} =\frac{p^{r^2}}{p^r}=p^r$$ 
and the fact that the identity endomorphism of $U(l)$ corresponds to the class $[l] \in K_0 \rep(l,l)$, we conclude that \eqref{eq:comp-2} is equal to $p^r \cdot \id_{U(l)}$. The proof follows now from the fact that the image of \eqref{eq:comp-1} (or \eqref{eq:comp-2}) under the functor \eqref{eq:functor3} (with $R=\bbZ[1/p]$) is an isomorphism in $\Hmo_0(k)_{\bbZ[1/p]}$.
\section{Proof of Theorem \ref{thm:main2}}\label{sec:proof-main2}
For technical reasons, we will replace the categories of separable $k$-algebras and $G$-sets by their skeletons. Clearly, this procedure preserves
the associated motivic categories up to equivalence.
We treat first the case of the category $\Cov'(G)$. Consider it as a {\em graph}, \ie as a category without units and composition. Let us start by constructing a graph isomorphism $Q : \Cov'(G) \stackrel{\simeq}{\to} \Sep(k)$. On objects we set $Q((S,A)):= U(A)$. This assignment establishes a bijection (because we are working with skeletal categories!) between the objects of $\Cov'(G)$ and the objects of $\Sep(k)$. Its inverse is given by $U(A) \mapsto (\Hom_{k\text{-}\mathrm{alg}}(Z(A),k_{\mathrm{sep}}), A)$, where $Z(A)$ stands for the center of $A$. Given objects $(S_1,A)$ and $(S_2,B)$ of $\Cov'(G)$ and $(s_1,s_2) \in S_1 \times S_2$, consider the following $G$-invariant map
$$ 
\delta'_{(s_1,s_2)}(t_1,t_2):= \begin{cases}
\mathrm{ind}_{(s_1,s_2)}(A^\op \otimes B) & \text{when $(t_1,t_2) \in G(s_1,s_2)$}\\
0 & \mathrm{otherwise}\,.
\end{cases}
$$
Note that these maps form a $\bbZ$-basis of $\Map^{G,A,B}(S_1 \times S_2, \bbZ)$. Making use of them, we set $Q(\delta'_{(s_1,s_2)}):=[I_{(s_1,s_2)}] \in K_0(A^\op \otimes B)$, where $I_{(s_1,s_2)}$ stands for the minimal ideal of the central simple $k_{(s_1,s_2)}$-algebra $(A^\op \otimes B)_{(s_1,s_2)}$. Thanks to Lemma \ref{lem:aux}(i) below (with $A$ replaced by $A^\op \otimes B$), we obtain an isomorphism 
$$ \Map^{G,A,B}(S_1\times S_2, \bbZ) \stackrel{\sim}{\too} \Hom_{\Sep(k)}(U(A),U(B)) \simeq K_0(A^\op \otimes B)\,.$$
This concludes the construction of the graph isomorphism $Q$. To every dg category $\cA$ we can associate the dg $k_{\mathrm{sep}}$-linear category $\cA \otimes k_{\mathrm{sep}}$ obtained by tensoring the cochain complexes of $k$-vector spaces $\cA(x,y)$ with $k_{\mathrm{sep}}$. This assignment is functorial on $\cA$ and, as proved in \cite[\S7]{Artin}, gives rise to an additive symmetric monoidal functor $-\otimes k_{\mathrm{sep}}: \NChow(k) \to \NChow(k_{\mathrm{sep}})$. Making use of it, consider now the following diagram of graphs
\begin{equation}\label{eq:diagram-last}
\xymatrix{
\Cov'(G) \ar[d]_-{(S,A) \mapsto S} \ar[r]^-Q_-{\simeq} & \Sep(k) \ar[d]^-\phi \ar[r]^-{-\otimes k_{\mathrm{sep}}} & \Sep(k_{\mathrm{sep}}) \ar[r]^-{\simeq} & \CSep(k_{\mathrm{sep}}) \ar@{=}[d] \\
\Cov(G) \ar[r]_-{\eqref{eq:equivalence-new}}^-{\simeq} & \CSep(k) \ar[rr]_-{-\otimes k_{\mathrm{sep}}} && \CSep(k_{\mathrm{sep}})\,,
}
\end{equation}
where $\phi$ is the unique graph morphism making the left-hand side square commute. Given an object $(S,A)$ of $\Cov'(G)$, we have $U(A \otimes k_{\mathrm{sep}})\simeq U(k_S \otimes k_{\mathrm{sep}})$. Moreover, thanks to Lemma \ref{lem:aux}(ii) below (with $A$ replaced by $A^\op \otimes B$), the following equality
$$ [I_{(s_1,s_2)} \otimes k_{\mathrm{sep}}]=\mathrm{ind}_{(s_1,s_2)}(A^\op \otimes B) \cdot [k_{(s_1,s_2)}\otimes k_{\mathrm{sep}}]$$
holds in the Grothendieck group $K_0((A^\op \otimes B)\otimes k_{\mathrm{sep}})$. These two facts imply that the above diagram \eqref{eq:diagram-last} is commutative. Note that all the graph morphisms of the right-hand side rectangle are functors, except a priori $\phi$, and that $-\otimes k_{\mathrm{sep}}$ is moreover faithful. This implies that $\phi$ is also a functor. Using the left-hand side commutative square, we hence conclude that $\Cov'(G)$ is a well-defined additive symmetric monoidal category and that $Q$ is an equivalence of categories.

The case of the category $\Hecke'(G)$ is similar. Let us start by constructing a graph isomorphism $Q':\Hecke'(G) \stackrel{\simeq}{\to} \Cov'(G)$. On objects we set $Q'((H,A)):= (G/H,A)$. On morphisms we observe that \eqref{eq:maps1} restricts to an isomorphism 
$$ \Map^{A,B}(H\backslash G/K, \bbZ) \stackrel{\sim}{\too} \Map^{G,A,B}(G/H\times G/K,\bbZ)\,.$$
Consider now the following commutative diagram of graphs:
\begin{equation}\label{eq:com-diagram-last}
\xymatrix{
\Hecke'(G) \ar[d]_-{(H,A) \mapsto H} \ar[rr]^-{Q'}_-{\simeq} && \Cov'(G) \ar[d]^-{(S,A) \mapsto S} \\
\Hecke(G) \ar[rr]^-\simeq_-{H \mapsto G/H} && \Cov(G)\,.
}
\end{equation}
As above, \eqref{eq:com-diagram-last} allows us to conclude that $\Hecke'(G)$ is a well-defined additive category and that $Q'$ is an equivalence of categories.
\begin{lemma}\label{lem:aux}
Let $S$ be a finite $G$-set and $A$ an Azumaya algebra over $k_S$; see Notation \ref{not:new}. Given $s \in S$, let us write $I_s$ for the minimal ideal of the central simple $k_s$-algebra $A_s$. Under these notations, the following holds:
\begin{itemize}
\item[(i)] The Grothendieck group $K_0(A)$ is freely generated by the classes $[I_s]$ of the finitely generated projective right $A$-modules $I_s$, where $s$ runs through a set of representatives of the $G$-orbits in $S$.
\item[(ii)] The image of $[I_s]$ under the group homomorphism
$$ K_0(A) \stackrel{-\otimes k_{\mathrm{sep}}}{\too} K_0(A \otimes k_{\mathrm{sep}}) \stackrel{\mathrm{Morita}}{\simeq} K_0(k_S \otimes k_{\mathrm{sep}})$$
identifies with $\mathrm{ind}_s(A) \cdot [k_s \otimes k_{\mathrm{sep}}]$.
\end{itemize}
\end{lemma}
\begin{proof}
The Azumaya $k_S$-algebra $A$ is Morita equivalent to the product $\prod_s A_s$, where $s$ runs through a set of representatives of the $G$-orbits in $S$. Therefore, the proof of item (i) follows from the fact that the class $[I_s]$ is a generator of the Grothendieck group $K_0(A_s) \simeq \bbZ$. In what concerns item (ii), consider the following commutative diagrams
$$
\xymatrix{
K_0(A) \ar[d]_-{-\otimes_A A_{s'}} \ar[rr]^-{-\otimes k_{\mathrm{sep}}} && K_0(A \otimes k_{\mathrm{sep}}) \ar[d]_-{(-\otimes_A A_{s'})\otimes k_{\mathrm{sep}}} \ar[rr]^-{\mathrm{Morita}}_-{\sim} && K_0(k_S \otimes k_{\mathrm{sep}}) \ar[d]^-{(-\otimes_{k_S}k_{s'})\otimes k_{\mathrm{sep}}} \\
K_0(A_{s'}) \ar[rr]_-{-\otimes k_{\mathrm{sep}}} && K_0(A_{s'} \otimes k_{\mathrm{sep}}) \ar[rr]_-{\mathrm{Morita}}^-{\sim} && K_0(k_{s'} \otimes k_{\mathrm{sep}})\,,
}
$$
where $s'$ runs through a set of representatives of the $G$-orbits in $S$. Under the canonical isomorphism $K_0(k_S \otimes k_{\mathrm{sep}}) \simeq \prod_{s'} K_0(k_{s'} \otimes k_{\mathrm{sep}})$, the right vertical map corresponds to the projection into the $s'^{\mathrm{th}}$-factor. Moreover, $[I_s]\otimes_A A_{s'}$ is equal to $[I_s] \in K_0(A_{s'})$ when $s'=s$ and equal to $0$ when $s'\neq s$. Therefore, the proof of item (ii) follows from the fact that the image of $[I_s] \in K_0(A_s)$ under the bottom horizontal composition agrees with $\mathrm{ind}(A_s) \cdot [k_s \otimes k_{\mathrm{sep}}]$.
\end{proof}
\section{Proof of Proposition \ref{prop:ring}}
Under the equivalence $\Sep(k) \simeq \Cov'(G)$ of Theorem \ref{thm:main2}, $\CSA(k)$ identifies with the full subcategory of $\Cov'(G)$ consisting of those objects $(S,A)$ with $S$ a fixed singleton $\{s\}$. Consequently, we obtain the following identifications
\begin{equation}\label{eq:identification-first}
\Hom_{\CSA(k)}(U(A_i),U(A_j))\simeq \mathrm{ind}(A^\op_i \otimes A_j) \cdot \bbZ\,,
\end{equation}
under which the composition law of $\CSA(k)$ corresponds to multiplication. Since $\mu_{ij}:=\mathrm{ind}(A^\op_i \otimes A_j)$, the proof of item (i) follows then from the definition of $\Lambda(A_1, \ldots, A_n)$. The proof of item (ii) is standard and we leave it to the reader.
\section{Proof of Theorem \ref{thm:comprehensive} and Proposition \ref{prop:groupcase}}
Let $A_1, \ldots, A_n$ be central simple $k$-algebras as in Proposition \ref{prop:ring} and $\Lambda:=\Lambda(A_1,\ldots,A_n)$. We start by specializing Arnold's results \cite{Arnold} to our situation.
\begin{theorem} \label{thm:Arnoldthm}
{(see \cite[Thm.~I]{Arnold})} Let $P_1,\ldots,P_n \in \mathrm{Proj}(\Lambda)$ be the right $\Lambda$-modules given by the rows of the matrix representation \eqref{eq:matrix} of $\Lambda$, and $Q \in \mathrm{Proj}(\Lambda)$ an indecomposable right $\Lambda$-module. Then, for every $p \in \cP$ there exists an integer $\varrho_p\in \{1,\ldots, n\}$ such that $Q_{(p)}\simeq (P_{\varrho_p})_{(p)}$.
\end{theorem}
\begin{remark}\label{rk:rows}
Note that $P_1, \ldots, P_n$ are the images of $U(A_1), \ldots, U(A_n)$ under the equivalence of categories $\varphi$ of Proposition \ref{prop:ring}(ii).
\end{remark}
Arnold's result \cite[Thm.~I]{Arnold} also applies to each one of the rings $\Lambda_{(p)}$. In these cases, we obtain the following result:
\begin{proposition} \label{prop:fgen} The indecomposable finitely generated projective right $\Lambda_{(p)}$-modules are of the form $(P_i)_{(p)}$ with $i \in \{1, \ldots, n\}$.
\end{proposition}
Arnold's work \cite{Arnold} gives also rise to the following results:
\begin{proposition} \label{prop:localp} Assume given for any $p\in \cP$ a right $\Lambda_{(p)}$-module $P_p \in \mathrm{Proj}(\Lambda_{(p)})$ of $\bbZ_{(p)}$-rank $rn$. Then, there exists a right $\Lambda$-module\footnote{Thanks to Theorem \ref{thm:Arnoldcor}, the right $\Lambda$-module $P$ is moreover unique.} $P \in \mathrm{Proj}(\Lambda)$ such that $P_{(p)}=P_p$ for every $p \in \cP$.
\end{proposition}
\begin{proof} Making use of Proposition \ref{prop:fgen}, we can assume without loss of generality that $r=1$, i.e.\ that all the $P_p$'s are indecomposable. Given $p \in \cP$, let $\varrho_p\in \{1, \ldots, n\}$ be such that $(P_{\varrho_p})_{(p)}=P_p$; see Theorem \ref{thm:Arnoldthm}. We need to construct a right $\Lambda$-module $P \in \mathrm{Proj}(\Lambda)$ with the property that $P_{(p)}=(P_{\varrho_p})_{(p)}$ for every $p \in \cP$.
Using the matrix representation \eqref{eq:matrix} of $\Lambda$ and the above Remark \ref{rk:rows}, we observe that all the $P_p$'s may be viewed as submodules of finite index of the right
  $\Lambda$-module $[\bbZ\cdots \bbZ]$. It then suffices to
  take $P:=\bigcap_{p} (P_{\varrho_p})_{(p)}$, where the intersection
  takes place inside $[\bbQ\cdots \bbQ]$; see \cite[Prop.\ I.~5.2]{Fossum}.
\end{proof}

\begin{theorem} \label{thm:Arnoldcor}{(see \cite[Cor. II]{Arnold})} 
Given right $\Lambda$-modules $P,Q \in \mathrm{Proj}(\Lambda)$, one has $P\simeq Q$ if and only if $P_{(p)}\simeq Q_{(p)}$ for every $p \in \cP$.
\label{thm:isolambda}
\end{theorem}
\begin{corollary} \label{cor:sumcorollary}
Given right $\Lambda$-modules $P,Q \in \mathrm{Proj}(\Lambda)$, $Q$ is a direct summand of $\cP$ if and only if $Q_{(p)}$ is a direct summand of $P_{(p)}$ for every $p \in \cP$.
\end{corollary}
\begin{proof} We focus ourselves on the non-obvious implication. Assume that $Q_{(p)}$ is a direct summand of $P_{(p)}$ for every $p\in \cP$. For each such $p$ choose a complement $Q_{(p)}\oplus R_p\simeq P_{(p)}$. Using Proposition \ref{prop:localp}, we hence obtain a well-defined right $\Lambda$-module $R \in \mathrm{Proj}(\Lambda)$ such that $R_{(p)}=R_p$ for every $p\in \cP$.  Theorem \ref{thm:Arnoldcor} allows us then to conclude that $Q\oplus R\simeq P$.
\end{proof}
Given a prime number $p$, let us denote by $\Hmo_0(k)_{(p)}$ the
$\bbZ_{(p)}$-linear category $\Hmo_0(k)_{\bbZ_{(p)}}$, by $(-)_p$ the functor \eqref{eq:functor3} with $R=\bbZ_{(p)}$, and by
$U(-)_{(p)}$ the functor $U(-)_{\bbZ_{(p)}}$. Under these notations, we have the following combative diagram
\begin{equation}\label{eq:commutative-square}
\xymatrix{
\CSA(A_1,\ldots,A_n)^{\oplus,\natural}\ar[rr]_-{\simeq}^-{\varphi}\ar[d]_{(-)_{(p)}} &&\mathrm{Proj}(\Lambda)\ar[d]^{(-)_{(p)}}\\
\CSA(A_1,\ldots,A_n)^{\oplus,\natural}_{(p)}\ar[rr]^-{\simeq}_-{\varphi_p} &&\mathrm{Proj}(\Lambda_{(p)})\,,\\
}
\end{equation}
where $\varphi_p$ stands for the equivalence $\Hom(\oplus_i U(A_i)_{(p)},-)$; see Proposition \ref{prop:ring}(ii). In what follows, we write $\CSA(k)^{\oplus,\natural}$ from the closure of $\CSA(k)$ under finite direct sums and direct summands.
\begin{lemma} \label{lem:localiso}
Given noncommutative motives $M,M'\in\CSA(k)^{\oplus,\natural}$, one has $M \simeq M'$ if and only if $M_{(p)}\simeq M'_{(p)}$ for every $p \in \cP$.
\end{lemma}
\begin{proof} 
We focus ourselves on the non-obvious implication. Assume that $M_{(p)}\simeq M'_{(p)}$ for every $p \in \cP$. Since $M,M' \in \CSA(A_1, \ldots, A_n)^{\oplus, \natural}$ for suitable central simple $k$-algebras $A_1, \ldots, A_n$, the above commutative diagram \eqref{eq:commutative-square} implies that 
$$ \varphi(M)_{(p)}\simeq \varphi_p(M_{(p)})\simeq \varphi_p(M'_{(p)})\simeq\varphi(M')_{(p)} \in \mathrm{Proj}(\Lambda_{(p)})$$
for every $p \in \cP$. Using Theorem \ref{thm:Arnoldcor} and the fact that $\varphi$ is an equivalence of categories, we hence conclude that $M\simeq M'$.
\end{proof}
\begin{lemma}\label{lem:Brauer2} 
Given central simple $k$-algebras $A$ and $B$, the following holds:
\begin{itemize}
\item[(i)] For every $p\in\cP$, one has $U(A)_{(p)}=U(A^p)_{(p)}$.
\item[(ii)] One has $U(A)_{(p)} \simeq U(B)_{(p)}$ for every $p \in \cP$ if and only if $[A]=[B]$.
\end{itemize}
\end{lemma}
\begin{proof}
Recall from \cite[Thm.~2.1]{TV} that given a central simple $k$-algebra $C$ whose index is a prime power $p^r$, we have $U(k)_R \simeq U(C)_R$ for every commutative ring $R$ containing $1/p$. In particular, $U(k)_{(q)}\simeq U(C)_{(q)}$ for every prime number $q \neq p$. Consider the $p$-primary decomposition $A=\otimes_{p \in \cP} A^p$.
Since the functor $U(-)_{(p)}$ is
symmetric monoidal, we hence conclude that $U(A)_{(p)}\simeq
U(A^p)_{(p)}$. This proves item (i). In what concerns item (ii), the non-obvious implication follows from the combination of Lemma \ref{lem:localiso} with Equivalence \eqref{eq:Brauer}.
\end{proof}
\begin{remark}
It is possible to prove Lemma \ref{lem:Brauer2}(ii) without invoking the results of Arnold \cite{Arnold} (\ie\ Theorem \ref{thm:Arnoldcor}). We focus ourselves in the non-obvious implication. Assume that $[A]\neq [B]$ or equivalently that $U(A) \not\simeq U(B)$; see Equivalence \eqref{eq:Brauer}. Thanks to Lemma \ref{lem:Brauer2}(i), we can assume without loss of generality that $\mathrm{ind}(A)$ and $\mathrm{ind}(B)$ are powers of a prime number $p$. Hence, we obtain the identification
\begin{equation}\label{eq:identification1}
\Hom_{\CSA(k)}(U(A),U(B))\stackrel{\eqref{eq:identification-first}}{\simeq} \mathrm{ind}(A^\op \otimes B) \cdot \bbZ = p^s \bbZ
\end{equation}
for some integer $s \geq 1$. Under \eqref{eq:identification1}, the composition bilinear pairing
$$\Hom_{\CSA(k)}(U(A),U(B))\times \Hom_{\CSA(k)}(U(B),U(A)) \too \Hom_{\CSA(k)}(U(A),U(B))$$
identifies with the multiplication pairing $p^s\bbZ \times p^s\bbZ \to \bbZ$. The analogous composition pairing, with $U(-)$ replaced by $U(-)_{(p)}$, identifies also with the multiplication pairing $p^s\bbZ_{(p)} \times p^s\bbZ_{(p)} \to \bbZ_{(p)}$. Since $s \geq 1$, the element $1 \in \bbZ_{(p)}$ is not in the image of the latter pairing. This allows us to conclude that $U(A)_{(p)}\not\simeq U(B)_{(p)}$.
\end{remark}
\subsection*{Proof of Theorem \ref{thm:comprehensive}}
\subsubsection*{Item (i)} Given $M \in \CSA(k)^{\oplus, \natural}$, we need to prove that $M \in \CSA(k)^\oplus$. Clearly, we can assume that $M \in \CSA(A_1, \ldots, A_n)^{\oplus, \natural}$ for suitable central simple $k$-algebras $A_1, \ldots, A_n$. Making use of the equivalence of categories $\varphi$, we have $\varphi(M)\simeq Q_1\oplus \cdots \oplus Q_m$ with $Q_1, \ldots, Q_m \in \mathrm{Proj}(\Lambda)$ indecomposable right $\Lambda$-modules. Without loss of generality, we can also assume that $m=1$; let $Q:=Q_1$.

Now, recall from Theorem \ref{thm:Arnoldthm} that for every $p \in \cP$ there exists an integer $\varrho_p\in \{1, \ldots, n\}$ such that $Q_{(p)}\simeq (P_{\varrho_p})_{(p)}$. Consequently, we obtain the identifications
$$
\varphi_p(M_{(p)})\stackrel{(a)}{\simeq} Q_{(p)}\simeq(P_{\varrho_p})_{(p)}\stackrel{(b)}{\simeq} \varphi(U(A_{\varrho_p}))_{(p)}\stackrel{(c)}{\simeq}\varphi_p(U(A^p_{\varrho_p})_{(p)})\,, 
$$
where (a) follows from the commutative diagram \eqref{eq:commutative-square}, (b) from Remark \ref{rk:rows}, and (c) from Lemma \ref{lem:Brauer2}(i) and \eqref{eq:commutative-square}. Using the fact that $\varphi_p$ is an equivalence of categories, we hence conclude that $M_{(p)}\simeq U(A^p_{\varrho_p})_{(p)}$. Let $B:= \otimes_{p \in \cP} A^p_{\varrho_p}$. Since $M_{(p)}\simeq U(B)_{(p)}$ for every $p\in \cP$, Lemma \ref{lem:localiso} implies that $M \simeq U(B)$. As a consequence, $M \in \CSA(k)^\oplus$.

\subsubsection*{Item (ii)}  Every object in $\CSA(k)^\oplus$ is of the form $U(A_1) \oplus \cdots \oplus U(A_n)$. Therefore, the indecomposable objects must be of the form $U(B)$ with $B$ a central simple $k$-algebra. Since the endomorphism rings $\mathrm{End}_{\CSA(k)}(U(B))\simeq \bbZ$ have no non-trivial idempotents, we conclude that the objects $U(B)$ are indeed indecomposable.

\subsubsection*{Item (iii)}  Let $B$ be a central simple $k$-algebra satisfying the following condition: for every $p \in \cP$ there exists an integer $\varrho_p \in \{1, \ldots, n\}$ such that $[B^p]=[A^p_{\varrho_p}]$. We need to prove that $U(B)$ is an indecomposable direct summand of $U(A_1) \oplus \cdots \oplus U(A_n)$. Consider the equivalence of categories
$$ \varphi: \CSA(A_1, \ldots, A_n,B)^{\oplus,\natural} \simeq \mathrm{Proj}(\Lambda(A_1, \ldots, A_n,B))\,.$$
Let us denote by $P_1, \ldots, P_n, Q$ the images of $U(A_1), \ldots, U(A_n), U( B)$ under $\varphi$. Under these notations, we have the following identifications
$$ Q_{(p)}\stackrel{(a)}{\simeq} \varphi_p(U(B^p)_{(p)}) \stackrel{(b)}{\simeq} \varphi_p(U(A^p_{\varrho_p})_{(p)})\stackrel{(c)}{\simeq}\varphi(U(A_{\varrho_p}))_{(p)} = (P_{\varrho_p})_{(p)}\,,$$
where (a) and (c) follow from the commutative diagram \eqref{eq:commutative-square} and  Lemma \ref{lem:Brauer2}(i), and $(b)$ from Lemma \ref{lem:Brauer2}(ii). Corollary \ref{cor:sumcorollary} implies then that $Q$ is a direct summand of $P_1\oplus \cdots \oplus P_n$. Using the fact that $\varphi$ is an equivalence of categories, we hence conclude that $U(B)$ is a direct summand of $U(A_1) \oplus \cdots \oplus U(A_n)$. Finally, as explained in item(ii), $U(B)$ is moreover indecomposable.

Now, let $M$ be an indecomposable direct summand of $U(A_1) \oplus \cdots \oplus U(A_n)$. Proceeding as in the proof of item (i), we conclude that $M\simeq U(B)$ with $B= \otimes_{p \in \cP}A^p_{\varrho_p}$. Clearly, for every $p \in \cP$ there exists an integer $\varrho_p \in \{1, \ldots, n\}$ such that $[B^p]=[A^p_{\varrho_p}]$. This concludes the proof.

\subsubsection*{Item (iv)}
We start by proving (a) implies (b). Clearly, condition (a) implies that 
\begin{equation}\label{eq:equivalence-p-last}
U(A_1)_{(p)}\oplus \cdots \oplus U(A_n)_{(p)}\simeq U(B_1)_{(p)} \oplus \cdots \oplus U(B_m)_{(p)}\end{equation}
for every $p \in \cP$. Consider the equivalence of categories
\begin{equation}\label{eq:equivalence-p}
\CSA(A_1, \ldots, A_n, B_1, \ldots, B_m)_{(p)} \stackrel{\varphi_p}{\simeq} \mathrm{Proj}(\Lambda(A_1, \ldots, A_n, B_1, \ldots, B_m)_{(p)})\,.
\end{equation}
Since the right-hand side of \eqref{eq:equivalence-p} is a Krull-Schmidt category, we hence conclude that $n=m$ and that there exists a permutation $\sigma_p$ (which depends on $p$) such that $U(B_i)_{(p)}\simeq U(A_{\sigma_p(i)})_{(p)}$ for every $1\leq i \leq n$. Thanks to Lemma \ref{lem:Brauer2}, the latter condition is equivalent to condition (b).

Let us now prove the converse implication. Thanks once again to Lemma \ref{lem:Brauer2}, condition (b) implies that $n=m$ and that the above isomorphism \eqref{eq:equivalence-p-last} holds for every $p \in \cP$. By applying Lemma \ref{lem:localiso} to the left and right-hand side of \eqref{eq:equivalence-p-last}, we hence obtain condition (a). This concludes the proof.

\subsection*{Proof of Proposition \ref{prop:groupcase}}
As explained in the proof of Theorem \ref{thm:comprehensive}(ii), the objects $U(A_i)$ are indecomposable. We now prove the converse. Let $M$ be an indecomposable object in $\CSA(A_1, \ldots, A_n)^{\oplus, \natural}$. Note that in the proof of Theorem \ref{thm:comprehensive}(i) we can assume without loss of generality that the Brauer classes $\{[A_1],\ldots, [A_n]\}$ of the suitable central simple $k$-algebras $A_1, \ldots, A_n$ form a subgroup of $\mathrm{Br}(k)$. Therefore, as in {\em loc. cit.}, we conclude that $M\simeq U(B)$ with $B=\otimes_{p \in \cP} A^p_{\varrho_p}$. Since the class $[A^p_{\varrho_p}]$ belongs to the cyclic subgroup of $\mathrm{Br}(k)$ generated by $[A_{\varrho_p}]$, $[B] \in \{[A_1], \ldots, [A_n]\}$. Hence, the proof follows now automatically from Equivalence \eqref{eq:Brauer}.

\section{Proof of Proposition \ref{prop:graded-ring}}
The objects $\{U(A_i)\}_{1\leq i \leq n}$ and $\{\Sigma(A_1, \ldots, A_n)([A_i])\}_{1\leq i \leq n}$ form a set of generators of the additive categories $\CSA(A_1, \ldots, A_n)^\oplus$ and $\mathrm{Proj}_{\mathrm{gr}}(\Sigma(A_1, \ldots, A_n))$, respectively. Since the functor $U(A_i) \mapsto \Sigma(A_1, \ldots, A_n)([A_i])$ is not only additive but also symmetric monoidal, it suffices to show that the induced homomorphisms
$$ \Hom(U(A_i),U(A_j)) \too \Hom(\Sigma(A_1, \ldots, A_n)([A_i]), \Sigma(A_1, \ldots, A_n)([A_j]))$$
are invertible. As explained above, the left-hand side identifies with $\mathrm{ind}(A_i^\op \otimes A_j) \cdot \bbZ$. In what concerns the right-hand side, it identifies with 
\begin{eqnarray*}
&  &  \Hom(\Sigma(A_1, \ldots, A_n)([k]), \Sigma(A_1, \ldots, A_n)([A_i^\op \otimes A_j])) \label{eq:identification}\\
&  \simeq & \Sigma(A_1, \ldots, A_n)([A_i^\op \otimes A_j])_{[k]} = \mathrm{ind}(A^\op_i \otimes A_j) \cdot \bbZ \nonumber
\end{eqnarray*}
because $\Sigma(A_1, \ldots, A_n)([A_i])$ is a strongly dualizable object of $\mathrm{Proj}_{\mathrm{gr}}(\Sigma(A_1, \ldots, A_n))$ with dual $\Sigma(A_1, \ldots, A_n)([A_i^\op])$. This completes the proof.
\section{Proof of Theorem \ref{thm:computation}}
Let us prove first the particular case where $\mathrm{per}(A)$ is a prime power $p^r$. Since $(1-t)(1+t + \cdots+ t^{p^r-1})=(1-t^{p^r})$, we need then to show that the assignment $t \mapsto U(A)$ gives rise to a ring isomorphism 
\begin{equation}\label{eq:homo-aux}
\bbZ[t]/\langle (1-t^{p^r})\rangle \stackrel{\sim}{\too} K_0(\CSA(k,A, A^{\otimes 2}, \ldots, A^{\otimes (p^r-1)})^\oplus)\,.
\end{equation}
Thanks to the cancellation property of Corollary \ref{cor:cancellation}, the elements of the right-hand side of \eqref{eq:homo-aux} are formal differences (not just equivalence classes)
\begin{eqnarray}\label{eq:element}
[\oplus_j U(A^{\otimes s_j})]-[\oplus_{j'}U(A^{\otimes s_{j'}})] && 0 \leq s_j, s_{j'} \leq p^r-1\,.
\end{eqnarray}
This implies that the above homomorphism \eqref{eq:homo-aux} is surjective. Now, recall that \eqref{eq:element} is trivial if and only if there exists an isomorphism 
\begin{equation}\label{eq:isom-aux}
\oplus_j U(A^{\otimes s_j})\simeq \oplus_{j'} U(A^{\otimes s_{j'}})\,.
\end{equation}
Since $\mathrm{per}(A)$ and $\mathrm{ind}(A)$ have the same prime
factors, $A^p=A$ and $A^q=k$ for every prime number $q \neq p$. Hence,
by combining Theorem \ref{thm:comprehensive}(iv) with the equivalence
\eqref{eq:Brauer}, we conclude that \eqref{eq:isom-aux} holds if and
only if $j=j'$ and there exists a permutation $\sigma$ such that
$U(A^{\otimes s_{j'}})\simeq U(A^{\otimes s_{\sigma (j)}})$ for every
$j$. This implies that the above homomorphism \eqref{eq:homo-aux} is
moreover injective, and therefore an isomorphism.

Let us now prove the general case where $\mathrm{per}(A)=\prod_i p_i^{r_i}$. As above, the assignment $t \mapsto U(A)$ gives rise to a surjective ring homomorphism 
\begin{equation}\label{eq:surjective}
\eta:\bbZ[t]/\langle (1-t^{\mathrm{per}(A)})\rangle \twoheadrightarrow K_0(\CSA(k,A,A^{\otimes 2}, \ldots, A^{\otimes (\mathrm{per}(A)-1)})^\oplus)\,.
\end{equation}
For every prime power $p_i^{r_i} \in \mathrm{per}(A)$ consider the following commutative diagram\footnote{Thanks to Theorem \ref{thm:comprehensive}(iv) that the right-hand side vertical homomorphism is well-defined.}:
$$
\xymatrix{
\bbZ[t]/\langle (1-t^{\mathrm{per}(A)})\rangle \ar[d]_-{t \mapsto t} \ar@{->>}[rr]^-\eta && K_0(\CSA(k,A,A^{\otimes 2}, \ldots, A^{\otimes (\mathrm{per}(A)-1)})^\oplus) \ar[d]^-{A \mapsto A^{p_i}} \\
\bbZ[t]/\langle (1-t^{p_i^{r_i}})\rangle \ar[rr]^-{\sim}_-{t \mapsto U(A^{p_i})} && K_0(\CSA(k,A^{p_i}, (A^{p_i})^{\otimes 2}, \ldots, (A^{p_i})^{\otimes(p_i^{r_i}-1)})^\oplus)\,.
}
$$
The lower horizontal homomorphism is an isomorphism (as proved above) and the kernel of the left-hand side vertical homomorphism is given by $\langle(1-t^{p_i^{r_i}}) \rangle$. The commutativity of the above square implies then that $\mathrm{Ker}(\eta) \subseteq \bigcap_i \langle  (1-t^{p_i^{r_i}}) \rangle$. Let us now prove the converse inclusion, or equivalently that the intersection of the kernels of the right-hand side vertical homomorphisms is trivial. Recall from above that the elements of the right-hand side of \eqref{eq:surjective} are formal differences
\begin{eqnarray}\label{eq:element2}
[\oplus_j U(A^{\otimes s_j})]-[\oplus_{j'}U(A^{\otimes s_{j'}})] && 0 \leq s_j, s_{j'} \leq \mathrm{per}(A)-1\,.
\end{eqnarray}
On one hand, \eqref{eq:element2} is trivial if and only if there
exists an isomorphism $\oplus_j U(A^{\otimes s_j}) \simeq \oplus_{j'}
U(A^{\otimes s_{j'}})$. On the other hand, as explained above,
\eqref{eq:element2} belongs to the kernel of the right-hand side
vertical morphisms if and only if $j=j'$ and there exists a
permutation $\sigma_{p_i}$ (which depends on $p_i$) such that
$U((A^{p_i})^{\otimes s_{j'}})\simeq U((A^{p_i})^{\otimes
  s_{\sigma_{p_i}(j)}})$ for every $j$. Thanks to Theorem
\ref{thm:comprehensive}(iv), these two conditions are equivalent. Hence,
$\mathrm{Ker}(\eta) = \bigcap_i \langle (1-t^{p_i^{r_i}})
\rangle$. Finally, making use of the equalities $1-t^{p_i^{r_i}} =
(1-t)(1+ t + t^2+ \cdots + t^{p_i^{r_i}-1})$, we conclude that
$$  \mathrm{Ker}(\eta)=\bigcap_i \langle  (1-t^{p_i^{r_i}}) \rangle = \langle(1-t) \prod_i (1 +t + t^2 + \cdots + t^{p_i^{r_i}-1})\rangle\,.$$
This concludes the proof.
\section{Proof of Proposition \ref{prop:implication}}
Making use of Theorem \ref{thm:comprehensive}(iv)(b) and the fact that every permutation $\sigma_p$ can be written as a composition of transpositions, it suffices to prove the following claim: given central simple $k$-algebras $D_1=\otimes_{q \in \cP} D_1^q$ and $D_2=\otimes_{q \in \cP} D_2^q$, the relations \eqref{eq:relations2} gives rise to an isomorphism between the following objects ($\cP':=\cP\backslash\{p\}$):
\begin{eqnarray}
E(D_1^p \otimes(\otimes_{q \in \cP'} D_1^q))\oplus E(D_2^p \otimes(\otimes_{q \in \cP'} D_2^q)) \label{eq:searched-1} \\
E(D_2^p \otimes(\otimes_{q \in \cP'} D_1^q))\oplus E(D_1^p \otimes(\otimes_{q \in \cP'} D_2^q)) \label{eq:searched-2} \,.
\end{eqnarray}
The relation \eqref{eq:relations2} applied to the central simple $k$-algebras $A:= (D_1^p)^\op \otimes D_2^p$, $B:=(\otimes_{q \in \cP'}D_1^q)\otimes (\otimes_{q \in \cP'}D_2^q)$ and $C:=D_1$, implies that \eqref{eq:searched-1}-\eqref{eq:searched-2} are isomorphic. This proves item (i). In what concerns item (ii), note that by tensoring \eqref{eq:searched-1}-\eqref{eq:searched-2} with the object $E(D_1^\op)$, we obtain:
\begin{eqnarray}
E(k) \oplus E((D_1^p)^\op\otimes D_2^p \otimes (\otimes_{q \in \cP'}D_1^q)^\op \otimes (\otimes_{q \in \cP'}D_2^q)) \label{eq:given-1} \\
E((D_1^p)^\op \otimes D_2^p) \oplus E((\otimes_{q \in \cP'}D_1^q)^\op\otimes (\otimes_{q \in \cP'}D_2^q))\,. \label{eq:given-2}
\end{eqnarray}
Hence, relation \eqref{eq:relations2} applied to the central simple $k$-algebras $A := (D_1^p)^\op \otimes D_2^p$, $B:=(\otimes_{q \in \cP'}D_1^q)^\op\otimes (\otimes_{q \in \cP'}D_2^q)$ and $C:=k$ implies that \eqref{eq:given-1}-\eqref{eq:given-2} are isomorphic. Finally, using the fact that the object $E(\cD_1^\op)$ is $\otimes$-invertible, we conclude that \eqref{eq:searched-1}-\eqref{eq:searched-2} are also isomorphic. This proves item (ii).
\section{Proof of Theorem \ref{thm:Flag}}
Recall from \cite[\S1]{MP} the construction of Merkurjev-Panin's motivic category $\underline{\cC}$ and of the symmetric monoidal functors $\Phi: \SmProj(k)^\op \to \underline{\cC}$ and $\Psi: \mathrm{sep}(k) \to \underline{\cC}$, where $\mathrm{sep}(k)$ stands for the category of separable $k$-algebras.
\begin{proposition}\label{prop:computation}
We have an isomorphism (see Notation \ref{not:notation2})
\begin{equation}\label{eq:iso-aux}
\Phi(\mathrm{Flag}(d_1, \ldots, d_m;A))\simeq {\mathrm{deg}(A) \choose d_1 \cdots d_m}_{\Psi(A)}\,.
\end{equation}
\end{proposition}
\begin{proof}
Recall that $G:=\operatorname{Gal}(k_{\text{sep}}/k)$. We start by recalling some results from Panin's work \cite{Panin}.  Let $\tilde{H}$ be a split
semi-simple simply connected algebraic group (defined over $k$) and $\tilde{T}\subset\tilde{P}\subset \tilde{H}$ a split maximal torus and a parabolic subgroup. The corresponding representation rings are denoted by $R(\tilde{T}),
R(\tilde{P})$, and $R(\tilde{H})$, respectively. Let
  $\tilde{Z}$ be the center
of $\tilde{H}$ and $H:=\tilde{H}/\tilde{Z}$ the corresponding adjoint
group. We write $\Ch=\Hom(\tilde{Z},\bbG_m)$ for the character group. The representation rings
introduced above are canonically $\Ch$-graded. A representation $V$
has degree $\cX:\tilde{Z}\r \bbG_m$ if $\tilde{Z}$ acts on $V$
through the character $\cX$. \def\Gal{\operatorname{Gal}}Fix an element $\gamma\in
H^1(G,H(k_{\text{sep}}))$. It gives rise to a canonical group
homomorphism $\beta_\gamma:\Ch\r \Br(k)$
which sends $\cX:\tilde{Z}\r \bbG_m$ to the image of $\gamma$ under the composed homomorphism 
$$
H^1(G,H(k_{\text{sep}}))\xrightarrow{\partial}  H^2(G,\tilde{\cZ}(k_{\text{sep}}))\xrightarrow{\cX}H^2(G,\bbG_m(k_{\text{sep}}))
=\Br(k)\,.
$$
Let $\cF:=\tilde{H}/\tilde{P}$. 
 Since the $\tilde{H}$-action
on $\cF$ factors through $H$ we have a corresponding twisted
homogeneous space ${}_\gamma\cF$. 
For any $\Ch$-homogeneous basis $\theta_1,\ldots,\theta_q$ of $R(\tilde{P})$ over $R(\tilde{H})$,
Panin constructs an isomorphism in the category $\underline{\cC}$
\begin{equation}
\label{eq:panin}
\bigoplus_{i=1}^q \Psi(\beta_{\gamma}(|\theta_i|))\simeq \Phi({}_\gamma\cF)\,,
\end{equation}
where $|\theta_i|\in \Ch$ is the degree of $\theta_i$. This
construction follows from the combination of \cite[Thm.~6.7]{Panin} with the functor $F_\gamma$
introduced in the proof of \cite[Lem.~6.5]{Panin}.

\medskip
\def\Gl{\operatorname{Gl}}
\def\Sl{\operatorname{Sl}}
We now specialize\footnote{It would be more convenient to use $\tilde{H}=\Gl_n$
but this is not a semi-simple group so it does not fall literally 
under Panin's setting \cite{Panin}. 
} the above constructions to $\tilde{H}=\Sl_n$, where
$n:=\mathrm{deg}(A)$.  In this case, $\tilde{\cZ}$ consists of
the constant diagonal matrices with entries the $n$-roots of unity. Hence, $\Ch$ has a canonical generator given by the
inclusion of $\tilde{\cZ}$ inside the diagonal matrices of
$\Gl_n$. This implies that canonically $\Ch=\bbZ/n\bbZ$. Let
$\tilde{P}\subset \tilde{H}$ be the standard parabolic subgroup of
elements preserving the flag $(0\subset k^{d_1}\subset
k^{d_1+d_2}\subset\cdots\subset k^{d_1+\cdots+d_n}=k^n)$.  In this
case, we have ${}_\gamma\cF=\operatorname{Flag}(d_1,\ldots,d_m;A)$,
where $[A]=\beta_\gamma(\bar{1})$.  Let $\tilde{T}\subset \tilde{P}$
be the maximal torus of diagonal matrices. If we denote by $t_i$ the
$i^{\mathrm{th}}$ diagonal entry, then $t_i$ is an element of
$R(\tilde{T})=\Hom(\tilde{T},\bbG_m)$.  Under the above choices, the
associated Weyl groups $\cW_{\tilde{P}}$ and $\cW_{\tilde{H}}$ of
$\tilde{P}$ and $\tilde{H}$ are given, respectively, by $S_{d_1}\times
\cdots\times S_{d_m}$ and $S_n$. Therefore, we have the equalities
\begin{eqnarray*}
R(\tilde{T})&=&\tilde{R}(\tilde{T})/(t_1\ldots t_n-1)\\ 
R(\tilde{P})&=&R(\tilde{P})^{\cW_{\tilde{P}}}=\tilde{R}(\tilde{P})/(\sigma^{(1)}_{d_1}\sigma^{(2)}_{d_2}\ldots \sigma^{(m)}_{d_m}-1)\\
R(\tilde{H})&=&R(\tilde{H})^{\cW_{\tilde{H}}}=\tilde{R}(\tilde{H})/(\sigma_{n}-1)\,,
\end{eqnarray*}
where $\tilde{R}(\tilde{T})$, $\tilde{R}(\tilde{P})$, and $\tilde{R}(\tilde{H})$, are given respectively by 
\begin{eqnarray*}
\bbZ[t_1,\ldots,t_n]  & \bbZ[\sigma_1^{(1)},\ldots,\sigma_{d_1}^{(1)},\ldots \sigma_1^{(m)},\ldots,
\sigma_{d_m}^{(m)}] & \bbZ[\sigma_1,\ldots,\sigma_n]
\end{eqnarray*}
for the appropriate symmetric functions $\sigma^{(j)}_i$ and $\sigma_k$ in $t_1, \ldots, t_n$. The $\Ch$-grading on the above is obtained from the $\bbZ$-grading on $\tilde{R}(\tilde{T})$ given by $|t_i|=1$. 

Now, fix a $\bbZ$-graded basis $\tilde{\theta}_1,\ldots,\tilde{\theta}_q$ of 
$\tilde{R}(\tilde{P})$ over $\tilde{R}(\tilde{H})$. Tensoring $-\otimes_{\tilde{R}(\tilde{T})} R(\tilde{T})$
yields a $\Ch$-graded basis $\theta_1,\ldots,\theta_q$ of $R(\tilde{P})$ over $R(\tilde{H})$. Since $A^{\otimes n}$ is Morita equivalent to $k$ we hence obtain the following equalities:
\begin{equation}\label{eq:equalities}
\Psi(\beta_\gamma(|\theta_i|))=\Psi(\beta_\gamma(\bar{1})^{|\tilde{\theta}_i|})=\Psi([A]^{|\tilde{\theta}_i|})=\Psi(A)^{\otimes |\tilde{\theta_i}|}\,.
\end{equation}
Finally, by combining \eqref{eq:equalities} into \eqref{eq:panin} and with the Poincar{\'e}
series computation
\begin{equation*}
\label{eq:gaussian}
\sum_i t^{|\tilde{\theta}_i|}=\left( {n\atop d_1\cdots d_m}\right)_t\,,
\end{equation*}
we obtain the desired isomorphism \eqref{eq:iso-aux}. This concludes the proof.
\end{proof}
As proved in \cite[Thm.~6.10]{twisted}, there exists an additive fully-faithful symmetric monoidal functor $\Theta: \underline{\cC} \to \NChow(k) \subset \Hmo_0(k)$ making the diagrams commute:
$$
\xymatrix{
\SmProj(k)^\op\ar[d]_-{\Phi} \ar[rr]^-{X \mapsto \perf_\dg(X)} &&  \dgcat(k) \ar[d]^-U & \mathrm{sep}(k) \ar[d]_-\Psi \ar[rr]^-{A\mapsto A} && \dgcat(k) \ar[d]^-U \\
\underline{\cC} \ar[rr]_-{\Theta} && \Hmo_0(k) & \underline{\cC} \ar[rr]_-{\Theta}  && \Hmo_0(k)\,.
}
$$
As a consequence, the searched motivic decomposition \eqref{eq:motivic-decomp-flag} is obtained by applying the functor $\Theta$ to the isomorphism \eqref{eq:iso-aux}. This achieves the proof.

\appendix

\section{Cyclic sieving phenomenon}\label{app:combinatorial}
Let $C$ be a cyclic group of order $n$ acting on a finite set $X$ and $p(t) \in \bbZ[t]$ a polynomial with integer coefficients. Recall from \cite{Reiner} that the triple $(X,p(t), C)$ {\em exhibits the cyclic sieving phenomenon} if $|X^c|=p(\omega)$ for every $c \in C$ and root of unit $\omega$ having the same order as $c$. Intuitively speaking, the polynomial $p(t)$ works as a generating function for the set $X$.
\begin{example}\label{ex:cyclic}
Given non-negative integers $m \leq n$, let $X$ be the set of all $m$-element subsets of $\{1,2, \ldots, n\}$ and $p(t)$ the {\em Gaussian polynomial}
$$ {n \choose m}_t := \frac{(1-t^n)(1-t^{n-1})\cdots(1-t^{n-m+1})}{(1-t)(1-t^2)\cdots (1-t^m)} \in \bbN[t]\,.$$
If $c\in C$ acts on $X$ by cycling the elements of a $m$-subset modulo $n$, then the triple $(X,p(t),C)$ exhibits the cyclic sieving phenomenon; see \cite[Thm.~1.1(b)]{Reiner}. 
\end{example}
\begin{lemma}
Let $(X,p(t),C)$ be a triple exhibiting the cyclic sieving phenomenon. Given a divisor $l|n$ and a coprime integer $(i,l)=1$, we have the congruence relation 
\begin{equation}\label{eq:modulo}
p(t^i) \cong p(t) \,\,\mathrm{modulo}\,\,(t^l-1)\,.
\end{equation}
\end{lemma}
\begin{proof}
It suffices to show that $p(\omega^i)=p(\omega)$ for every $\omega$ such that $\omega^l=1$. Since by hypothesis $(i,l)=1$, the roots of unit $\omega$ and $\omega^i$ have the same order. Consequently, $p(\omega^i)=|X^c|=p(\omega)$ for some $c \in C$. This achieves the proof.
\end{proof}
\begin{corollary}\label{cor:sieving}
Given non-negative integers $m \leq n$, an additive symmetric monoidal category $(\cC,\otimes, {\bf 1})$, and an object $b \in \cC$ such that $b^{\otimes l}\simeq {\bf 1}$, we have an induced isomorphism ${n \choose m}_{b^{\otimes i}} \simeq {n \choose m}_b$; see Notation \ref{not:notation2}.
\end{corollary}
\begin{proof}
Since $b^{\otimes l}\simeq {\bf 1}$, we have the following evaluation map
\begin{eqnarray*}
\bbN[t]/(t^l=1) \too \mathrm{Iso}(\cC) && p(t) \mapsto p(b)\,.
\end{eqnarray*}
Consequently, the proof follows from the above congruence \eqref{eq:modulo}.
\end{proof}
\section{Noncommutative motives of dg Azumaya algebras}\label{app:dgAzumaya}
In this appendix we assume that $k$ is a commutative ring; let $s$ be the number of components of the associated affine $k$-scheme $\mathrm{Spec}(k)$. Recall from \cite{ICM-Keller,Buenos,Additive} that all constructions and results of \S\ref{sec:dg}-\ref{sec:NCmotives} hold also in this generality; simply replace the tensor product by its derived version $-\otimes^{\bf L}-$. 
\subsection*{DG Azumaya algebras}
A dg $k$-algebra $A$ is called a {\em dg Azumaya algebra} if:
\begin{itemize}
\item[(i)] The underlying complex of $k$-modules is a compact generator of $\cD(k)$.
\item[(ii)] The canonical morphism $A^\op \otimes^{\bf L} A \to {\bf R}\Hom(A,A)$ in $\cD(k)$ is an isomorphism.
\end{itemize}
\begin{example}\label{ex:Azumaya}
\begin{itemize}
\item[(i)] The ordinary Azumaya algebras (see \cite{Grothendieck}) are the dg Azumaya algebras whose underlying complex is $k$-flat and concentrated in degree zero.
\item[(ii)] When $k$ is a field, every dg Azumaya algebra is isomorphic in the homotopy category $\Hmo(k)$ to an ordinary Azumaya algebra; see \cite[Prop.~2.12]{Toen}.
\item[(iii)] For every non-torsion {\'e}tale cohomology class $\alpha \in H^2_{\mathrm{et}}(\mathrm{Spec}(k),\bbG_m)$ there exists a dg Azumaya algebra $A_\alpha$ (representing this class $\alpha$) which is {\em not} isomorphic in $\Hmo(k)$ to an ordinary Azumaya algebra; see \cite[page~584]{Toen}. Unfortunately, the construction of $A_\alpha$ is not explicit. 
\end{itemize}
\end{example}
The {\em derived Brauer group $\dBr(k)$ of $k$} is the set of isomorphism classes of dg Azumaya algebras in $\Hmo(k)$. The group structure is induced by the derived tensor product. 
B.~To{\"e}n constructed in \cite[Cor.~3.8]{Toen} an injective map
\begin{equation}
\label{eq:toen}
\psi:\dBr(k)\too H^1_{\operatorname{et}}(\mathrm{Spec}(k),\bbZ) \times H^2_{\operatorname{et}}(\mathrm{Spec}(k),\bbG_m) 
\end{equation}
We will now describe this map in down-to-earth terms, avoiding the language of derived stacks. For notational reasons, we will follow the geometric setting.

Let $X=\mathrm{Spec}(k)$, $\cL(X)$ the set of locally constant functions $X \to \bbZ$, and $\DPic(X)$ the derived Picard group of $X$ (which we consider as a $2$-group). As proved by Rouquier-Zimmermann in \cite[\S3]{Rouquier}, we have the following equivalence
\begin{eqnarray}
\label{eq:dpiciso}
\cL(X)\times \Pic(X)\stackrel{\sim}{\too} \DPic(X) && (\underline{n},L)\mapsto L[\underline{n}]\,,
\end{eqnarray}
where $L[\underline{n}]$ is such that $L[\underline{n}]_{|U}:=L_{|U}[n_U]$ for every connected component $U \subset X$. 

Now, let $A$ be a dg Azumaya algebra over $X$. Following \cite[Prop.\ 1.14]{Toen}, there exists an {\'e}tale cover $f:Y \to X$ such that $f^\ast(A)\simeq \REnd_Y(P)$ for some perfect complex on $Y$. Let us write $(Y/X)_\bullet$ for the associated hypercovering with $Y_n:=(Y/X)_n:=Y^{\times_X^n}$. Under these notations, we have the following isomorphisms:
\begin{equation}\label{eq:isom-B}
\REnd_{Y_2}(\pr^\ast_1(P))\simeq \pr^\ast_1\REnd_Y(P)\simeq\pr_2^\ast\REnd_Y(P)\simeq \REnd_{Y_2}(\pr^\ast_2(P))\,.
\end{equation}
In what follows, we will use the notations $(-)_{i\cdots j}:=\pr_{i\cdots j}^\ast(-)$. Thanks to Morita theory, \eqref{eq:isom-B} is induced from an isomorphism 
\begin{eqnarray}
\label{eq:startiso}
\gamma:L\otimes^{\bf L}_{Y_2} P_1\stackrel{\sim}{\too} P_2 && L\in \DPic(Y)\,,
\end{eqnarray}
which is unique up to multiplication with an element of $\cO_{Y_2}^\ast$. By pulling-back $\gamma$ to $Y_3$, in three different ways, we obtain once again by Morita theory an isomorphism $\phi:L_{23}\otimes^{\bf L}_{Y_3} L_{12}\stackrel{\sim}{\to} L_{13}$ in $\DPic(Y)$ making the following diagram commute:
\begin{equation}
\label{eq:foundation}
\xymatrix@C=3em@R=2em{
L_{23}\otimes^{\bf L}_{Y_3} L_{12}\otimes^{\bf L}_{Y_3} P_1\ar[d]_{\phi\otimes\Id} \ar[rr]^-{\Id\otimes \gamma_{12}}&&  L_{23}\otimes^{\bf L}_{Y_3} P_2\ar[d]^{\gamma_{23}}\\
L_{13} \otimes^{\bf L}_{Y_3} P_1\ar[rr]_{\gamma_{13}}&& P_3\,.
}
\end{equation}
The morphism $\phi$ satisfies the standard cocycle condition when pulled-back to $Y_4$. Moreover, $\phi$ is well-defined up to an obvious type of coboundary. Let $\bar{\phi}$ be the corresponding
equivalence class. We call $(Y,L,\bar{\phi})$ a set of \emph{Picard data} on $Y$. 

Thanks to the above isomorphism \eqref{eq:dpiciso}, we have $L=L'[\underline{n}]$ with $L'\in \Pic(Y)$. Moreover, the isomorphism $\phi$ implies that $\underline{n}_{12}+
\underline{n}_{23}=\underline{n}_{13}$ on $Y_2$. Hence, $\underline{n}$
defines an element of $\check{H}^1((Y/X)_\bullet,\bbZ)$. Now, let $g:Z\r Y_2$ be an {\'e}tale cover\footnote{Of course, the ``\'etale'' cover can be taken to be a Zariski cover.} such that $g^\ast(L')\simeq \cO_Z$ in $\Pic(Z)$. In what follows, we fix such a trivialization. Let $S_{1,\bullet}$ be the
truncated hypercovering $Z\mathbin{\vcenter{\vbox{\hbox{$\r$}\nointerlineskip\hbox{$\r$}}}}Y\mathbin{\r}
X$, $S_\bullet$ the coskeleton $\cosk S_{1,\bullet}$, and $g:S_\bullet\r (Y/X)_\bullet$ the
induced map of hypercoverings. Using the chosen trivialization, the following morphism (deduced from $\phi$)
\[
g^\ast(\phi):g^\ast(L_{23}')\otimes^{\bf L}_{Z}  g^\ast(L_{12}')
\too g^\ast(L_{13}')
\]
is given by multiplication with an element $\tilde{\phi}\in \Gamma(Z,\cO^\ast_Z)$.  Since $\tilde{\phi}$ still satisfies the cocycle
condition, it defines an element in $\check{H}^2(S^\bullet,\bbG_m)$; 
 one checks that this element does not depend on the chosen trivialization neither on the isomorphism $\gamma$.
 
 The map \eqref{eq:toen} can now be explicitly described as the image of $(\underline{n},\tilde{\phi})$
under
\[
 \check{H}^1((Y/X)_\bullet,\bbZ)\times \check{H}^2(S_\bullet,\bbG_m)\too H^1_{\text{et}}(X,\bbZ)\times {H}_{\text{et}}^2(X,\bbG_m)\,.
\]
\begin{lemma} 
\label{lem:bclass} Let $A$ be a dg Azumaya algebra over $X$ and $(Y,L,\bar{\phi})$ a set of Picard data for $A$. Then, 
$A$ is trivial if and only if $(Y,L,\bar{\phi})$ satisfies the following
condition: there exists an element $K\in \DPic(Y)$ and an isomorphism $\theta:L\stackrel{\sim}{\to} K_2\otimes^{\bf L}_{Y_2} K_1^{-1}$ making the following diagram commute:
\begin{equation}
\label{eq:cond2}
\xymatrix{
L_{23}\otimes^{\bf L}_{Y_3} L_{12}\ar[d]_{\theta_{23}\otimes \theta_{12}}\ar[rr]^{\phi} && L_{13}\ar[d]^{\theta_{13}}\\
(K_3\otimes^{\bf L}_{Y_3} K_2^{-1})\otimes^{\bf L}_{Y_3} (K_2\otimes^{\bf L}_{Y_3} K_1^{-1})
\ar[rr]_-{\mathrm{canonical}} && K_3\otimes^{\bf L}_{Y_3} K_1^{-1}\,.
}
\end{equation}
\end{lemma}
\begin{proof} Let $(Y,L,\overline{\phi})$ be a Picard data for $A$ satisfying the above condition(s). In this case, we can replace the $L$ in the isomorphism $\phi:L \otimes^{\bf L}_{Y_2}P_1 \stackrel{\sim}{\to} P_2$ by $K_1 \otimes^{\bf L}_{Y_2}K_2^{-1}$. Via the above procedure, using the lower row of \eqref{eq:cond2}, one observes that $\psi(A)$ is trivial. Therefore, the injectivity of \eqref{eq:toen} implies that $A$ is trivial. 

Let us now prove the converse. Assume that $A$ is trivial. In this case, $A\simeq \REnd_X(Q)$ with $Q$ a perfect complex
on $X$. Consider the isomorphisms
\begin{equation}\label{eq:iso-B2-1}
\REnd_Y(P)\simeq f^\ast(A)=\REnd_X(f^\ast(Q))\,.
\end{equation}
Thanks to Morita theory, \eqref{eq:iso-B2-1} is induced from an isomorphism $f^\ast (Q)\simeq K\otimes^{\bf L}_Y P$ with $K \in \DPic(Y)$. Since $(f^\ast(Q))_1\simeq Q_{12} \simeq (f^\ast(Q))_2$ on $Y_2$, we hence obtain an isomorphism $K_1\otimes^{\bf L}_{Y_2} P_1\simeq K_2\otimes^{\bf L}_{Y_2} P_2$. This implies that $L\simeq K_1\otimes^{\bf L}_{Y_2} K_2^{-1}$. The commutativity of \eqref{eq:cond2} follows from Morita theory. This concludes the proof.
\end{proof}
\begin{definition}
Given a dg Azumaya algebra $A$ over $X$, let $\deg(A):=\sqrt{\mathrm{rank}(A)}$, where $\mathrm{rank}(A)$ is defined by the usual alternating sum of ranks in each degree.
\end{definition}
We consider $\deg(A)$ as a locally constant function on $X$.
Note that the local {\'e}tale triviality of $A$ implies that $\deg(A)$ takes integer values. The following result plays a key role in the case of ordinary Azumaya algebras; see \cite[\S IV Thm.~6.1]{Knus}. 
\begin{theorem} 
\label{thm:trivbrclass}
The dg Azumaya algebra $
A^{\otimes \deg(A)}
$ is trivial.
\end{theorem}
\begin{proof} Let $d:=\deg A$ and $B:=A^{\otimes d}$. If $(Y,L,\overline{\phi})$ is a set of Picard data for $A$, we need to show that the set of Picard data $(Y,L^{\otimes d},\overline{\phi}^{\otimes d})$ for $B$ satisfies the conditions of Lemma \ref{lem:bclass}. By taking the determinant of \eqref{eq:startiso} we obtain an isomorphism
\begin{eqnarray*}
\det(\gamma):L^{\otimes d}\otimes^{\bf L}_{Y_2} \det(P_1)\stackrel{\sim}{\to} \det(P_2) & \Leftrightarrow & \det(\gamma):L^{\otimes d}\stackrel{\sim}{\to} \det(P_2)\otimes^{\bf L}_{Y_2} \det(P_1)^{-1}\,.
\end{eqnarray*}
Similarly, by first taking the determinant of \eqref{eq:foundation}, and then tensoring the result with $\det P_1^{-1}$, we obtain the following commutative diagram:
\[
\xymatrix@C=0.01em@R=1.5em{
L_{23}^{\otimes d}\otimes^{\bf L}_{Y_3} 
L_{12}^{\otimes d}\ar@{.>}[rrrrdd]\ar[d]_{\phi^{\otimes d}\otimes\Id} \ar[rrrr]^-{\Id\otimes \det(\gamma_{12})}&&&&  L_{23}^{\otimes d}\otimes^{\bf L}_{Y_3} \det(P_2)\otimes^{\bf L}_{Y_3} \det(P_1)^{-1}\ar[dd]^{\det(\gamma_{23})\otimes \Id}\\
L_{13}^{\otimes d} \ar[d]_-{\det(\gamma_{13})}&& &&\\
\det(P_3)\otimes^{\bf L}_{Y_3} \det(P_1)^{-1}\ar[rrrr]_-{\sim}&&&& (\det(P_3)\otimes^{\bf L}_{Y_3} \det(P_2))^{-1}\otimes^{\bf L}_{Y_3}
(\det(P_2)\otimes^{\bf L}_{Y_3} \det(P_1)^{-1})
}
\]
Note that the lower triangle is precisely \eqref{eq:cond2} with $L$ replaced by $L^{\otimes d}$
and $K$ by $\det(P)$. This concludes the proof.
\end{proof}
The following result sheds some new light on dg Azumaya algebras.
\begin{theorem}\label{thm:rank}
Let $A$ be a dg Azumaya algebra of rank $(r_1,\ldots, r_s)$. When $r_1, \ldots, r_s \neq 0$, $A$ is isomorphic in  $\Hmo(k)$ to an ordinary Azumaya algebra.
\end{theorem}
\begin{proof}
We may assume without loss of generality that $\mathrm{Spec}(k)$ is indecomposable. In this case, $s=1$; let $r:=r_1$. Thanks to Theorem \ref{thm:trivbrclass}, the class $\psi(A)\in H^1_{\operatorname{et}}(\mathrm{Spec}(k),\bbZ) \times H^2_{\operatorname{et}}(\mathrm{Spec}(k),\bbG_m)
$ is $r$-torsion. Making use of the short exact sequence of constant sheaves $0\to \underline{\bbZ}\to\underline{\bbQ} \to \underline{\bbQ}/\underline{\bbZ}\to 0$, we observe that $H^1_{\text{et}}(\mathrm{Spec}(k),\bbZ)$ is contained in $H^1_{\text{et}}(\mathrm{Spec}(k),\bbQ)$. In particular, it is torsion-free. This hence implies that $\psi(A)\in H^2_{\text{et}}(\mathrm{Spec}(k),\bbG_m)_{\text{tors}}$. Finally, making use of Gabber's result \cite[Thm II.1]{Gabber}, we conclude that $A$ is isomorphic in $\Hmo(k)$ to an ordinary Azumaya algebra.
\end{proof}
\begin{corollary}\label{cor:new}
Let $A_\alpha$ be a dg Azumaya algebra as in Example \ref{ex:Azumaya}(iii) of rank $(r_1,\ldots, r_s)$. Then, there exists an integer $i \in \{1, \ldots, s\}$ such that $r_i=0$.
\end{corollary}
Corollary \ref{cor:new} shows that the dg Azumaya algebras $A_\alpha$ associated to non-torsion {\'e}tale cohomology classes $\alpha \in H^2_{\mathrm{et}}(\mathrm{Spec}(k),\bbG_m)$ have always trivial rank.
\subsection*{Noncommutative motives}
By combining the above Theorem \ref{thm:rank} with \cite[Thm.~2.1]{TV}, we obtain the following computation:
\begin{corollary}\label{cor:Azumaya}
Given a dg Azumaya algebra $A$ as in Theorem \ref{thm:rank}, we have a canonical isomorphism $U(k)_R \simeq U(A)_R$ for every commutative ring $R$ containing $1/r$ with $r:=r_1 \times \cdots \times r_s$.
\end{corollary}
Intuitively speaking, Corollary \ref{cor:Azumaya} shows that the difference between the noncommutative motives of $A$ and $k$ is a torsion phenomenon. As the next result shows, this is {\em not} the case when we consider the dg Azumaya algebras of Example \ref{ex:Azumaya}(iii): 
\begin{theorem}\label{thm:Azumaya}
  Let $A$ be a dg Azumaya
  algebra which is not isomorphic in $\Hmo(k)$ to an ordinary Azumaya
  algebra. When $k$ is noetherian\footnote{It is likely that this hypothesis is superfluous.}, we have $U(k)_\bbQ \not\simeq
  U(A)_\bbQ$.
\end{theorem}
\begin{proof}
Let us assume that $U(k)_\bbQ \simeq U(A)_\bbQ$. Thanks to the construction of the $\bbQ$-linear category $\Hmo_0(k)_\bbQ$ (see \S\ref{sub:coefficients}) and to the fact that $A$ is smooth and proper (see \cite[Prop.~2.5]{Toen}), there exists a right $A$-module $P \in \cD_c(A)$, a right $A^\op$-module $Q \in \cD_c(A^\op)$, and positive integers $m,n >0$ satisfying the following equalities:
\begin{eqnarray*}
 [P \otimes_A^{\bf L} Q] = n\cdot[k] \in K_0\cD_c(k) &&     [Q \otimes^{\bf L} P] = m\cdot[A] \in K_0\cD_c(A^\op \otimes^{\bf L} A)\,.
\end{eqnarray*}
Now, choose an {\'e}tale extension $k \to k'$ making $A':= A\otimes_k k'$ and $k'$ isomorphic in $\Hmo(k')$; see \cite[Prop.~2.14]{Toen}. Once again by construction of the category $\Hmo(k)$, there exists a right $A'$-module $R \in \cD_c(A')$, a right $A'^\op$-module $S$, and isomorphisms 
\begin{eqnarray*}
S \otimes^{\bf L}_{k'} R \simeq A' \in \cD_c(A'^\op \otimes^{\bf L}A) && R\otimes^{\bf L}_{A'}S \simeq k' \in \cD_c(k')\,,
\end{eqnarray*}
giving rise the following equalities:
\begin{eqnarray*}
[(P'\otimes^{\bf L}_{A'} S)\otimes^{\bf L}_{k'}(R\otimes^{\bf L}_{A'} Q')] = n\cdot [k']  &&
 [(R\otimes^{\bf L}_{A'} Q') \otimes^{\bf L}_{k'} (P'\otimes^{\bf L}_{A'} S)]=  m\cdot [k']\,.
\end{eqnarray*}
This clearly implies that the rank of $P' \otimes_{A'} S$ is non-trivial. Thanks to Lemma \ref{lem:rank} below, $P'\otimes_{A'} S$ is a compact generator
of $\cD(k')$. Hence, since $S$ induces an isomorphism in $\Hmo(k')$ between $A'$ and $k'$, $P'$ is a compact generator of $\cD(A')$. Making use of Lemma \ref{lem:extension} below, we conclude moreover that $P$ is a generator
of $A$. We can therefore consider the dg Azumaya $k$-algebra $B:= {\bf R}\mathrm{End}_A(P)$. Note that $B$ is isomorphic in $\Hmo(k)$ to the dg Azumaya $k$-algebra $A$. Using the equalities
$$ \mathrm{rank}(B) = \mathrm{rank}({\bf R} \mathrm{End}_A(P)) = \mathrm{rank} ({\bf R} \mathrm{End}_{A'}(P')) = \mathrm{rank}({\bf R}\mathrm{End}_{k'}(P' \otimes^{\bf L}_{A'} S))\,,$$
we observe that the rank of $B$ is also non-trivial. The above Theorem \ref{thm:rank} hence implies that $B$ (and consequently $A$) is isomorphic in $\Hmo(k)$ to an ordinary Azumaya algebra. This contradiction achieves the proof.
\end{proof}
\begin{lemma} \label{lem:rank}
Let $k$ be a commutative noetherian ring and $P\in \cD_c(k)$. If $\mathrm{rank}(P)\neq0$, then $P$ is a (compact) generator of $\cD(R)$.
\end{lemma}
\begin{proof} If $\mathrm{rank}(P)\neq 0$, then the support $\mathrm{supp}(P)$ of $P$ is equal to $\Spec(k)$. It follows then from Nakayama's
lemma that $P\otimes^{\bf L}_k k(\mathfrak{p})\neq 0$ for every $\mathfrak{p}\in \Spec(k)$. Using Neeman's work
\cite[Thm 2.8]{Neeman}, we hence conclude that the triangulated localizing subcategory of $\cD(k)$ generated by $P$ agrees with $\cD(k)$.
\end{proof}
\begin{lemma} 
\label{lem:extension} Let $A$ be a dg $k$-algebra, $P \in \cD_c(A)$, and $k'/k$ a faithfully flat extension. If $P'$ is a generator of $\cD(A')$, then $P$ is a (compact) generator of $\cD(A)$.
\end{lemma}
\begin{proof} 
Let $Q \in \cD(A)$ such that $\RHom_A(P,Q)=0$. Since the extension $k'/k$ is flat and $P$ is a perfect complex, we have $\RHom_{A'}(P',Q')=\RHom_A(P,Q)'=0$. Consequently, $M'=0$. Using the fact that the extension $k'/k$ is faithfully flat, we hence conclude that $M=0$.
\end{proof}
\begin{remark}
Let $A$ be a dg Azumaya algebra. Similarly to the case of ordinary Azumaya algebras, we have the following equivalence of categories:
\begin{eqnarray*}\label{eq:equiv-induced}
\cD_c(k) \stackrel{\simeq}{\too} \cD_c(A^\op \otimes^{\bf L} A) && P \mapsto P \otimes^{\bf L} A\,. 
\end{eqnarray*}
Hence, since the Hochschild homology $HH_\ast(A)$ and the Hochschild cohomology $HH^\ast(A)$ of $A$ can be recovered from $\cD_c(A^\op \otimes^{\bf L}A)$, we obtain induced isomorphisms
\begin{eqnarray*}
k \simeq HH_\ast(k) \simeq HH_\ast(A) && k \simeq HH^\ast(k) \simeq HH^\ast(A)\,.
\end{eqnarray*}
Note that Theorem \ref{thm:Azumaya} implies that these isomorphisms are {\em not} motivic, \ie they are not induced from an isomorphism in $\Hmo_0(k)$.
\end{remark}



\begin{thebibliography}{00}

\bibitem{Amitsur} S.~Amitsur, {\em Generic splitting fields of central simple algebras}. Annals of Mathematics {\bf 62} (1955), no.~2, 8--43.

\bibitem{Andre} Y.~Andr{\'e}, {\em Une introduction aux motifs (motifs purs, motifs mixtes, p{\'e}riodes)}. Panoramas et Synth{\`e}ses {\bf 17}. Soci{\'e}t{\'e} Math{\'e}matique de France, Paris, 2004.

\bibitem{AK} Y.~Andr{\'e} and B.~Kahn, {\em Nilpotence, radicaux et structures mono{\"i}dales}. (French) Rend.~Sem.~Mat. Univ.~Padova {\bf 108} (2002), 107--291.

\bibitem{Arnold} D.~M.~Arnold, {\em A finite global Azumaya theorem in additive categories}. Proceedings of the AMS {\bf 91} (1984), no.~1, 25--30.

\bibitem{MT} M.~Bernardara and G.~Tabuada, {\em Motivic decompositions in the commutative and noncommutative world}. Available at arXiv:1303.3172.

\bibitem{BO} A.~Bondal and D.~Orlov, {\em Semiorthogonal decomposition for algebraic varieties}. Available at arXiv:alg-geom/9506012.

\bibitem{CR} C.~Curtis and I.~Reiner, {\em Representation theory of finite groups and associative algebras}. 
Pure and Applied Mathematics, Vol. XI Interscience Publishers, a division of John Wiley Sons, New York-London 1962 xiv+685 pp.


\bibitem{Fossum} R.~Fossum, {\em The divisor class group of a Krull domain}. Ergebnisse der Mathematik und ihrer Grenzgebiete, Band {\bf 74}. Springer-Verlag, New York-Heidelberg, 1973. viii+148 pp. 

\bibitem{Gabber} O.~Gabber, {\em Some theorems on Azumaya algebras}. The Brauer group (Sem., Les Plans-sur-Bex, 1980), 129--209.

\bibitem{Grothendieck} A.~Grothendieck, {\em Le groupe de Brauer I: Alg{\`e}bres d'Azumaya et interpr{\'e}tations diverses}. Dix Expos{\'e}s sur la Cohomologie des Sch{\'e}mas, 46--66. North-Holland, Amsterdam; Paris.

\bibitem{Gille} P. Gille and T. Szamuely, {\em Central simple algebras and Galois cohomology}. Cambridge Studies in Advanced Mathematics {\bf 101}. Cambridge University Press, Cambridge (2006).

\bibitem{Karpenko} N.~Karpenko, {\em Criteria of motivic equivalence for quadratic forms and central simple algebras}. Math. Ann. {\bf 317} (2000), 585--611.

\bibitem{ICM-Keller} B.~Keller, {\em On differential graded categories}. International Congress of Mathematicians (Madrid), Vol.~II,  151--190. Eur.~Math.~Soc., Z{\"u}rich (2006).

\bibitem{IAS} M.~Kontsevich, {\em Noncommutative motives}. Talk at the IAS on the occasion of the $61^{\mathrm{st}}$ birthday of Pierre Deligne (2005). Available at {\tt http://video.ias.edu/Geometry-and-Arithmetic}.    
    
\bibitem{Miami} \bysame, {\em Mixed noncommutative motives}. Talk at the Workshop on Homological Mirror Symmetry, Miami (2010). Available at {\tt www-math.mit.edu/auroux/frg/miami10-notes}.  

\bibitem{finMot} \bysame, {\em Notes on motives in finite characteristic}.  Algebra, arithmetic, and geometry: in honor of Yu. I. Manin. Vol. II,  213--247, Progr. Math., {\bf 270}, BirkhŠuser Boston, MA, 2009.  

\bibitem{Knus} M.-A. Knus and M.~Ojanguren, {\em Th{\'e}orie de la descente et alg{\`e}bres d'Azumaya}. Lecture Notes in Mathematics {\bf 389}. Springer-Verlag, Berlin-New York (1974).



\bibitem{Semisimple} M.~Marcolli and G.~Tabuada, {\em Noncommutative motives, numerical equivalence, and semi-simplicity}. American Journal of Mathematics {\bf 136} (2014), no. 1, 59--75.

\bibitem{Artin}  \bysame, {\em Noncommutative Artin motives}. Selecta Mathematica, {\bf 20} (2014), 315--358.

\bibitem{Kontsevich} \bysame, {\em Kontsevich's noncommutative 
numerical motives}, Compositio Mathematica {\bf 148} (2012), no. 6, 1811--1820.

\bibitem{Galois} \bysame, {\em Noncommutative numerical motives, Tannakian structures, and motivic Galois groups}. Available at arXiv:1110.2438. To appear in Journal of the EMS.

\bibitem{MP} A. Merkurjev and I.~Panin, {\em $K$-theory of algebraic tori and toric varieties}. $K$-Theory {\bf 12} (1997), no.~2, 101--143.

\bibitem{Panin} I.~Panin, {\em On the algebraic K-theory of twisted flag varieties}. $K$-Theory {\bf 8} (1994), no.~6, 541--585. 

\bibitem{Neeman} A.~Neeman, {\em The chromatic tower for $D(R)$}. Topology {\bf 31} (1992), no.~3, 519--532.

\bibitem{Reiner} V. Reiner, D. Stanton and D. White, {\em The cyclic sieving phenomenon}. J. Combin. Theory Ser. A {\bf 108} (2004), 17--50.

\bibitem{Rouquier} R. Rouquier,  A. Zimmermann, {\em Picard groups for derived module categories}.  Proc. London Math. Soc. {\bf 87} (2003), no.~1, 197--225.  

\bibitem{Scott} L.~Scott, {\em Integral equivalence of permutation representations}. Group theory (Granville, OH, 1992), 262--274, World Sci. Publ., River Edge, NJ, 1993. 

\bibitem{Swan} R.~G.~Swan, {\em Projective modules over grop rings and maximal orders}. Annals of Mathematics {\bf 76} (1962), no.~1, 55--61.

\bibitem{twisted} G.~Tabuada, {\em Additive invariants of toric and twisted projective homogeneous varieties via noncommutative motives}. Journal of Algebra {\bf 417} (2014), 15--38.

\bibitem{CvsNC} \bysame, {\em Chow motives versus noncommutative motives}. Journal of Noncommutative Geometry {\bf 7} (2013), no. 3, 767--786.

\bibitem{Buenos} \bysame, {\em A guided tour through the garden of noncommutative motives}. Clay Mathematics Proceedings, Volume {\bf 16} (2012), 259--276.   

\bibitem{Additive} \bysame, {\em Additive invariants of dg categories}. Int. Math. Res. Not. {\bf 53} (2005), 3309--3339.   

\bibitem{TV} G.~Tabuada and M. Van den Bergh, {\em Noncommutative motives of Azumaya algebras}. Available at arXiv:1307.7946. To appear in Journal of the Institute of Mathematics of Jussieu.

\bibitem{Toen} B.~To{\"e}n, {\em Derived Azumaya algebras an generators for twisted derived categories}. Invent. Math. {\bf 189} (2012), 581--652.

\bibitem{Weibel} C.~Weibel, {\em $K$-book. An introduction to algebraic $K$-theory}. Graduate Studies in Mathematics, {\bf 145}. American Mathematical Society, Providence, RI, 2013. xii+618 pp.

\end{thebibliography}
\end{document}
\end{proof}